\documentclass{article}
\usepackage{parskip}
\usepackage{blkarray}
\usepackage{multirow}
\usepackage[pdftex]{graphicx}
\usepackage{epstopdf}
\usepackage{amssymb}
\usepackage{amsmath}
\usepackage{amsthm}
\usepackage{amsbsy}
\usepackage{psfrag}
\usepackage{pstricks}
\usepackage{float}
\usepackage{amsmath,hyperref}
\usepackage{tikz}
\usepackage{microtype}
\usepackage{parskip}
\usepackage{fontenc}
\usepackage{amsmath,amsthm,wasysym,amssymb,mathrsfs}
\usepackage{mathtools,appendix,epsfig,epsf,color,subfigure}
\usepackage{hyperref,pgfplots}
\usepackage{pgfplotstable}
\usepackage{algpseudocode}
\usepackage{setspace,relsize,needspace,etoolbox}
\usepackage[section]{placeins}
\usepackage{booktabs}
\makeatletter
\preto{\@verbatim}{\topsep=-1.5pt \partopsep=-1pt }
\makeatother
\usepackage[morefloats=7]{morefloats}
\theoremstyle{plain}
\newtheorem{theorem}{Theorem}[section]
\newtheorem{lemma}[theorem]{Lemma}
\newtheorem{corollary}[theorem]{Corollary}
\newtheorem{example}{Example}
\theoremstyle{remark}
\newtheorem{remark}[theorem]{Remark}
\newtheorem{algorithm}[theorem]{Algorithm}
\theoremstyle{definition}
\newtheorem{definition}{Definition}[section]

\title{ A novel parallel approach for solving some free boundary value problems}
\author{Peeyush Singh$^1$ and Amboru Yalamanda$^2$\\
  \small \noindent $^1$,Department of Mathematics and Statistics,VIT-AP University,AP, India.\\
  \small \noindent $^1$E-mails: peeyush.singh@vitap.ac.in,peeyushs8@gmail.com\\
\small \noindent $^2$ Department of Mathematics and Statistics,VIT-AP University,AP, India.\\
\small \noindent $^2$E-mails: amboru.24phd7073@vitap.ac.in,susindra.amboru@gmail.com 
}
\date{}
\begin{document}

\maketitle

\begin{abstract}
This paper introduces a novel, efficient class of parallel direct and indirect iterative schemes to solve general obstacle and free boundary value problems.
The uniqueness of the solution for the direct parallel method is established under the assumption that the model problem yields an $M$-matrix. The convergence analysis of the indirect approach is predicated on minimizing the discrete energy functional at each iterative stage of the linear approximation process.
Under the aforementioned setting, we establish theoretical convergence results for both smooth and nonsmooth energy functionals defined over a convex set, in the sense of Ferris and Mangasarian \cite{ferris1994parallel}.
For the numerical computation of the one-dimensional obstacle problem, we adopt a direct generalized parallel approach based on the SPIKE algorithm \cite{polizzi2006parallel}. To accelerate convergence and reduce computational complexity, a fast recursive version of the scheme is generalized to the nonsmooth case. For two- and three-dimensional obstacle problems, we employ a directional splitting method that treats each directional subproblem as a one-dimensional obstacle minimization problem. By framing the energy functional minimization as a parabolic time-dependent problem and utilizing an Armijo time-stepping rule during the solution update process, we successfully obtain results for higher-dimensional obstacles. Additionally, this study explores the possibility of extending the algorithm into a constrained quadratic programming optimization solver. We also evaluate the framework on image deblurring phenomena under the aforementioned setting, successfully recovering the original images. Finally, numerical illustrations are provided to validate the theoretical results.
\end{abstract}
\noindent \textbf{MSC2020 Subject Classification:} 65K15, 90C33, 65Y05, 49J40 \\
\noindent \textbf{ACM CCS Concepts:} Computing methodologies~Parallel algorithms; Mathematics of computing Mathematical software

\section{Introduction}
In a free boundary problem (FBP), both the partial differential equation and its unknown domain boundary must be solved simultaneously. Such problems frequently arise in physics and engineering, particularly in modeling phase transition phenomena like the melting of ice into water or in tribology problem (see for example \cite{peeyush2020}) like generation of caviation phenomenon in two moving contacts. The “free” boundary represents the moving interface between distinct phases or regions, such as the liquid-solid interface (or liquid-gas interface). Another application of free boundary problems arises in financial mathematics specifically in American-style option pricing, where the option value is modeled using an evolutionary variational inequality as well as in game theory as a Nash equilibrium. The brief review and current development in the field is as follows.
A rigorous theoretical framework of the obstacle problem originated with Stampacchia and coworkers \cite{stampacchia1964, lions1967variational}. Later, Lions \cite{lions1969quelques} derived approximation results for the model problem using monotone operators and penalty methods. The regularity result for the model was first proved by Caffarelli \cite{caffarelli1977regularity}, who provided a deep analytical understanding of the model.
Subsequent numerical developments included finite element discretizations and error 
estimates \cite{brezzi1977error}, augmented Lagrangian and penalty approaches 
\cite{glowinski1984numerical}, and variational discretization techniques 
\cite{nitsche1977variational}. Furthermore, Br{\'e}zis and Sibony 
\cite{brezis1968equivalence} proposed an iterative scheme for monotone operators.
Subsequently, Scholz \cite{scholz1984numerical, scholz1986optimal} implemented this 
scheme within a penalty finite element approximation framework to derive sharp 
a priori error bounds.
Conforming finite element methods (FEM) are standard for obstacle problems. While 
linear ($\mathcal{P}_1$) elements offer reliability, accuracy near free boundaries 
improves with quadratic ($\mathcal{P}_2$) variants in 2D \cite{bartels2015} and 3D. 
A posteriori error estimation and adaptive refinement advanced via localized 
estimators \cite{veeser2001}, pointwise error control and barrier sets 
\cite{nochetto2003}, and optimal convergence proofs for adaptive FEM \cite{braess2007}. 
Parallel developments in Discontinuous Galerkin (DG-FEM) methods yielded optimal 
a priori and a posteriori bounds \cite{gudi2014,weiss2010,ayuso2021}, alongside 
bubble-enriched $\mathcal{P}_2$ implementations in 3D \cite{gaddam2017}.
Parallel to these developments, several independent frameworks leveraging optimization 
theory have been proposed for solving obstacle problems. For instance, Friedlander 
and Tseng \cite{friedlander2007exact} introduced the exact regularization of convex 
programs, while Goldstein and Osher \cite{goldstein2009split} developed the split 
Bregman method for $L^{1}$-regularized problems. In the context of compressed sensing 
and $L^{1}$ optimization, Cai et al. \cite{cai2009linearized} applied linearized 
Bregman iterations, building upon the foundational advancements made by Cand{\`e}s 
and Donoho \cite{candes2006robust}. Additionally, Chambolle and Pock 
\cite{chambolle2011first} introduced highly efficient primal-dual algorithms. More 
recently, the dynamical functional particle method (DFPM) was developed to reformulate 
obstacle problems as first-order dynamical systems \cite{edvardsson2020dynamical}, 
alongside various least-squares finite element methods (LSFEM) \cite{bochev2009least}.
Iterative solvers became essential for practical computation. Hinterm{\"u}ller, Ito, and Kunisch \cite{hintermuller2003primal} introduced the primal-dual active set method, demonstrating its equivalence to a specific semismooth Newton-type algorithm.
Wang \cite{wang2002finite} provided finite element error bounds for quadratic 
elements under strong regularity assumptions, whereas Wang et al. 
\cite{wang2018two} proposed a two-level finite element algorithm that refines 
solutions near the free boundaries to achieve nearly optimal error bounds. 
Various domain decomposition methods were explored in the literature for solving the obstacle problem on parallel computers \cite{singh2020total}. Basic principal of these decomposition based on partitioning of global spatial domain into subdomains assigned to separate processors via MPI. Projected successive relaxation(PSOR) is well known iterative method for solving the obsatcle problem. Extending these idea on parallel computer  many people used PSOR via Red-Black ordering. These ordering breaks data dependencies enabling concurrent cell updates within decoupled subsets.
Badaya used Schwarz domain decomposition methods and solve localized obstacle problems, while primal-dual active set (PDAS) strategies reduce inactive regions to standard linear systems for faster solving. Additionally, he proved a convergence of a domain decomposition algorithm and provide convergence rate which based on minimizing quadratic functionals in Hilbert spaces. Various one and two level domain decompostion methods have been studied last couple of decades((for more review in this work see 
\cite{smith2004domain, toselli2005domain, quarteroni1999domain}). In 2003 Badaya and coworkers \cite{badea2003convergence} gave first explicit convergence rate estimates for one and two level Schwarz method for variational inequalities. Apart from these developments, some authors employed multi-grid methods for solving the model in the form of the full approximate scheme (FAS) \cite{brandt1977multi} 
or projected algebraic multigrid (PAMG) \cite{reisinger2007projected} to achieve optimal $\mathcal{O}(N)$ algorithmic complexity. In this process spatial domain decomposition scales operations across grid levels, but encounters communication bottlenecks at very coarse hierarchies.
Few authors \cite{he2009parallel, tao2012inexact} introduced splitting augmented Lagrangian methods for structured monotone variational inequalities with operators composed of two or three separable parts. The primary benefit of this method is that followers can execute their individual subproblems concurrently.
%
Recently, several authors \cite{glowinski2016alternating} have utilized operator splitting 
via the alternating direction method of multipliers (ADMM) to reformulate the 
elliptic state as a pseudo-transient parabolic problem.
They applied the splitting frameworks (like ADMM) decouple complex spatial dimensions into independent, explicit directional sweeps, eliminating global linear system updates and matching GPU multithreading models perfectly.
A few authors have also endeavored to obtain a discretized solution for the continuous variational inequality that minimizes energy over an admissible set , as well as for its discretized algebraic complementarity system.
This is realized by distributing parts of the stiffness matrix across concurrent parallel processing units. 
The current work is higly motivated from the work of A. Sameh and coworkers \cite{Spring2020SPIKE,polizzi2006parallel} on the develpment of efficient banded solvers SPIKE algorithm. The present study is adopted here due to its exceptional ability to solve large-scale banded linear systems in parallel while balancing communication overhead and computational complexicity cost. The SPIKE algorithm serves as a powerful high-performance computing (HPC) alternative to traditional sequential factorization such as LU decomposition which struggle with scalability issue on modern multiprocessor systems due to their sequential dependencies. These major communication bottlenecks can be easily handle through divide and conquer domain decomposition strategy. The solver offers superior high-performance computing scalability, minimal communication costs, massive speedup, and exceptional algorithmic flexibility. However, its application to constraint optimization has not yet been explored.
The general obstacle problem is a classical problem in the mathematical theory of partial differential equations (PDEs) and variational inequalities. It models physical situations, such as an elastic membrane clamped at the boundary and pushed from below by a rigid obstacle. However, the application of SPIKE dense banded solvers to constraint optimization such as these obstacle type problems has not yet been explored. These solvers have not been extensively investigated for constraint minimization problems, which opens an active research direction to reexamine the method and analyze its development as a commercial optimization software tool.
\subsection{The classical obstacle model }
Let $\Omega$ be a bounded domain in $\mathbb{R}^n$ ($n \in \{1, 2, 3\}$) with a smooth boundary $\partial\Omega$. Given the pde model parameters consisting of an obstacle $\psi_{ob} \in L^2(\Omega)$, a boundary profile $g \in H^{1/2}(\partial\Omega)$, and a forcing term $f \in L^2(\Omega)$, we define the set of admissible functions in the form of
\begin{equation}
    \mathcal{K} = \{ v \in H^1(\Omega) : v|_{\partial\Omega} = g \text{ and } v \geq \psi_{ob} \text{ a.e. in } \Omega \}.
    \label{eq:1}
\end{equation}
To solve the obstacle problem, we minimize an energy functional over the admissible set.
We seek a solution $u \in K$ that satisfies
\begin{equation}
     \min_{v \in \mathcal{K}}\mathcal{J}(v):=\int_{\Omega} \left( \frac{1}{2} |\nabla v|^2 - f v \right) dx
     \label{eq:2}
\end{equation}
Alternatively, this is equivalent to the \textbf{variational inequality} formulation: find a solution $u \in \mathcal{K}$ such that for all $v \in \mathcal{K}$:
\begin{equation}
    \int_{\Omega} \nabla u \cdot \nabla(v - u) dx \geq \int_{\Omega} f (v - u) dx
    \label{eq:3}
\end{equation}
Once we enforce regularizition conditions across both domain then it allows the obstacle problem to be stated in its pointwise complementarity form (KKT conditions).
A solution $u$ naturally partitions the domain $\Omega$ into a contact zone (or coincidence set) $\Omega_C = \{ x \in \Omega : u(x) = \psi_{ob}(x) \}$ and a non-contact zone $\Omega_N = \{ x \in \Omega : u(x) > \psi_{ob}(x) \}$.\\
The complementarity form states that solution $u$ must satisfy the following conditions almost everywhere in $\Omega$
\begin{align}
    u &\geq \psi_{ob} \label{eq:4} \\
    -\Delta u - f &\geq 0 \label{eq:5} \\
    (u - \psi_{ob}).(-\Delta u - f) &= 0 \label{eq:6} \\
    u &= g \quad \text{ on } \partial\Omega \label{eq:7}
\end{align}
To solve the obstacle problem numerically, we discretize the domain $\Omega$ using a mesh or grid (e.g., using Finite Difference or Finite Element methods). Let $N$ be the number of interior grid points. The discrete equivalent to minimizing the energy functional turns out a constrained \textbf{Quadratic Programming (QP)} problem
\begin{equation}
    \min_{\mathbf{u} \in \mathbb{R}^N} \left( \frac{1}{2} \mathbf{u}^T A \mathbf{u} - \mathbf{f}^T \mathbf{u} \right) \quad \text{subject to constrained } \mathbf{u} \geq \boldsymbol{\psi_{ob}},
    \label{eq:8}
\end{equation}
where $\mathbf{u}, \boldsymbol{\psi} \in \mathbb{R}^N$ are unknown solution and known obstacle respectively.
Using the discrete Karush-Kuhn-Tucker (KKT) optimization conditions, the matrix problem reformulated to the continuous complementary form. We seek a solution $\mathbf{u} \in \mathbb{R}^N$ such that
\begin{align}
    \mathbf{u} - \boldsymbol{\psi_{ob}} &\geq \mathbf{0} \label{eq:9} \\
    A\mathbf{u} - \mathbf{f} &\geq \mathbf{0} \label{eq:10}\\
    (\mathbf{u} - \boldsymbol{\psi_{ob}}).(A\mathbf{u} - \mathbf{f}) &= 0 \label{eq:11}
\end{align}
In this curreent work, we consider a parallel algorithm approach for a class of obstacle/free boundary value problems. 
The reminder paper is orgnized as follow: Section 2 present a mathematical background and an overview of the parallel algorithm setup. Section 3 discuss some equivalent formulation, its mathematical justification using $M$-matrix assumptions and novel formulation of projected parallel algorithms(direct and indirect approach) of the model. Section 4 contains generalized discussion on convergence analysis of parallel algorithms.
Section 5 elaborate numerical findings and applications of the algorithm. The final section (Section 6) wraps up the study, restates the main findings, and may suggest limitations and directions for future work.
\section{Mathematical distributive description}
Let $\mathcal{V}$ be a Banach space that is reflexive and let $\mathcal{J}$ be  a convex functional defined as $\mathcal{J}:\mathcal{V} \rightarrow \mathbb{R}$. We examine the 
optimization problem of the form
\begin{align}
   \min_{v \in \mathcal{K}}\mathcal{J}(v):=\int_{\Omega} \left( \frac{1}{2} |\nabla v|^2 - f v \right) dx, \quad \mathcal{K} \subset \mathcal{V},
   \label{eq:12}
\end{align}
where $\mathcal{K}$ denotes a closed convex subset of $\mathcal{V}$.
In this section, we focuses on the case in which the space $\mathcal{V}$ can be decomposed into a direct sum of subspaces $\mathcal{V}_{i}$, in other words,
\begin{align}
 \mathcal{V}=\mathcal{V}_{1}\oplus \mathcal{V}_{2}\oplus--\oplus \mathcal{V}_{r}=\sum_{i=1}^{r} \mathcal{V}_{i}.
 \label{eq:13}
\end{align}
This means that for any $v \in \mathcal{V}$, there exists a unique 
$v_{i} \in \mathcal{V}_{i}$ such that
\begin{align}
 v= \sum_{i=1}^{r}v_{i}
 \label{eq:14}
\end{align}
We solve Eqn~\eqref{eq:5} by decomposing $\mathcal{K}$ into sum of $\mathcal{K}_{i} \subset \mathcal{V}_{i}, i=1,2,..,r$, i.e.,
\begin{align}
 \mathcal{K}=\mathcal{K}_{1}\oplus \mathcal{K}_{2}\oplus--\oplus \mathcal{K}_{r}=\sum_{i=1}^{r} \mathcal{K}_{i},
 \label{eq:15}
\end{align}
where each $\mathcal{K}_{i} $ can be written as 
\begin{align}
 \mathcal{K}_{i} = \mathcal{K}_{i,in} \oplus \mathcal{K}_{i,mid} \oplus \mathcal{K}_{i,out}.
 \label{eq:16}
\end{align}
Then the solution of Eqn~\eqref{eq:5} can be evaluated by solving subparts of the minimization problem as the part of these component sets of each 
$\mathcal{K}_{i}$, where $i=1,2,..,r.$ in sequentially or parallel.

In other words, we can reformulate the minimization problem equivalent two minimization problems namely
\begin{align}
 P_{1}:\textit{ Find } u \in \sum_{i=1}^{r} (\mathcal{K}_{i,in} \oplus \mathcal{K}_{i,out})  \textit{ such that } \mathcal{J}_{1}(u)= \min_{v \in \sum_{i=1}^{r} (\mathcal{K}_{i,in} \oplus \mathcal{K}_{i,out})} \mathcal{J}(v)
 \label{eq:17}
\end{align}
and
\begin{align}
 P_{2}: \textit{ Find } u \in \sum_{i=1}^{r} \mathcal{K}_{i,mid}  \textit{ such that } \mathcal{J}_{2}(u)= \min_{v \in \sum_{i=1}^{r} \mathcal{K}_{i,mid} } \mathcal{J}(v).
 \label{eq:18}
\end{align}
\section{The projected parallel algorithm}

%
%
%
%
%
This work focuses on solving equations ~\eqref{eq:9}--~\eqref{eq:11} on parallel processing systems. Before we solve the parallel problem, we set up the linear complemeatrity problem (LCP) as $r$ blocks of matrices stored across $r$ processors. We then subdivide the complemeatrity system along the main diagonal into $r$ linear sub complementarity problems. Each block has a size of $n$ such that $N=nr$, where $r$ is the total number of available processors.
A block substructure $A_{i}, (i=1,2,…,r)$ is constructed here for the original linear complemenatrity system ~\eqref{eq:9}--~\eqref{eq:11} to create $r$ linear complemeatrity subsystem. The linear complementarity problem Eqn~~\eqref{eq:1} is partitioned into
	\begin{equation}
		\mathbf{C}_{i} \mathbf{x}_{i-1} + \mathbf{A}_{i} \mathbf{x}_{i} + \mathbf{B}_{i} \mathbf{x}_{i+1} \geq \mathbf{f}_{i}, \quad i = 1, 2, \dots, r
		\label{eq:19}
	\end{equation}
	\begin{equation}
		\mathbf{x}_{i} \geq \mathbf{0}
		\label{eq:20}
	\end{equation}
	\begin{equation}
		\mathbf{x}_{i}. \left( \mathbf{C}_{i} \mathbf{x}_{i-1} + \mathbf{A}_{i} \mathbf{x}_{i} + \mathbf{B}_{i} \mathbf{x}_{i+1} - \mathbf{f}_{i} \right) = 0	
		\label{eq:21}
	\end{equation}
,where $\mathbf{A}_{i}$ is the $n \times n$ block diagonal coefficient matrix of each partition. $\mathbf{x}_{i}$ and $\mathbf{f}_{i}$ are  $n \times 1$ vectors and 
$\mathbf{C}_{i}$,$\mathbf{B}_{i}$ are accompanied left and right offdiagonal blocks of the size $n \times n$. Note that 
	\begin{align}
		\mathbf{C}_{1} &= \mathcal{O}_{n \times n}, \quad \mathbf{B}_{r} = \mathcal{O}_{n \times n} ; \label{eq:22}\\
		\mathbf{x}_{0}= \mathcal{O}_{n \times 1}, \quad \mathbf{x}_{r+1} = \mathcal{O}_{n \times 1};
		\label{eq:23}
	\end{align}
	\vspace{1em}
	\[
	\mathbf{C}_{i} = 
	\begin{bmatrix}
		0  & \hat{C}_{i} \\
		0 & 0
	\end{bmatrix}_{n \times n}, \quad
	\mathbf{B}_{i} = 
	\begin{bmatrix}
		0 & 0 \\
		\hat{B}_{i} & 0
	\end{bmatrix}_{n \times n}
	\]
	\vspace{1em}
	\[
	\mathbf{x}_{i} = \begin{bmatrix} x_{i,1}, \dots, x_{i,n} \end{bmatrix}^T, \quad
	\mathbf{f}_{i} = \begin{bmatrix} f_{i,1}, \dots, f_{i,n} \end{bmatrix}^T
	\]
	\vspace{1em}
, where $B_{i}$ is an upper triangular matrix and $C_{i}$ is a lower triangular matrix.
For each partition $r$, Eqn.~\eqref{eq:2} can be reformulated as:
	\begin{equation}
		\mathbf{A}_{i} \mathbf{x}_{i} \geq \mathbf{f}_{i} - 
		\begin{bmatrix}
			\mathbf{C}_{i} \mathbf{x}_{i-1} \\
			0 \\
			\vdots \\
			0 \\
			\mathbf{B}_{i} \mathbf{x}_{i+1}
		\end{bmatrix}_{n \times 1}
		:= \mathbf{f}_{*i}, \quad i = 1, \dots, r 
		\label{eq:24}
		\end{equation}
\begin{equation}
		\mathbf{x}_{i} \geq \mathbf{0} 
		\label{eq:25}
\end{equation}
\begin{equation}
    \mathbf{x}_{i}. \left( \mathbf{A}_{i} \mathbf{x}_{i} - \mathbf{f}_{*i} \right) = 0 
    \label{eq:26}
	\end{equation}
Now, if each matrix $\mathbf{A}_{i}$ is invertible, then the partitioned 
linear complementarity systems (LCS)~\eqref{eq:24}--~\eqref{eq:26} completely decouple, except for the top and bottom $m$ blocks. This decoupling makes the system 
well-suited for solving the LCS on concurrent processors independently, 
inspiring a novel decomposition method for the original matrix~$\mathbf{A}$.
Our main idea in this article is to provide a novel, equivalent formulation 
of minimizing functionals based on a monotonic linear transformation of the 
original problem. In other words, we will provide two equivalent 
functionals, $J_{1}(u)$ and $J_{2}(u)$, which share the same set of 
minimizers as $J(u)$ (i.e., the same function $u$ minimizes both 
functionals).
To support the convergence analysis of our main algorithm, we introduce 
essential preliminary results.
\begin{theorem}
 Let $A$ be a $n \times n$ monotone matrix whose off diagonal entries are non positive. If for a given vector $x_{1}$ with the property $Ax_{1} \ge b$, then $x_{1} \ge A^{-1}b $.
\end{theorem}
\begin{proof}
 The proof of this result is a special case of the theorem shown in Varga \cite{varga2009matrix}.
\end{proof}
\begin{theorem}
 Let $A$ be a $n \times n$ nonsingular monotone matrix whose off diagonal entries are non positive. Then solving equation ~\eqref{eq:22} is equivalent solving the complemeatrity system of the form
 	\begin{equation}
	   	\begin{cases}
   		SX \geq G \textit { such that } DG=F \\
   		X \geq 0 \\
   		(SX-G)X=0
   	\end{cases}
   	\label{eq:27}
   \end{equation} 
,where matrix $D$ is a collection of  $r$-partioned $A_{i}$ diagonal blocks (see Figure 1), in other words $$D=diag(A_{1},.., A_{r})$$
and 
matrix $S$ is a collection of $r$-partioned $I_{i}$ identity matrices along the diagonal block
$$\tilde{I}=diag(I_{1},.., I_{r})$$ 
corresponding accompanied left and right blocks namely $V_i$ and $W_i$ are computed approximately from system of equation given below (see Figure 2)
\begin{equation}
  A_i[V_i,W_i]=
   \begin{bmatrix}
   	0 & C_i \\
   	. & 0   \\
   	. & .  \\
   	. & .  \\
   	0 & .  \\
   	B_i & 0  \\
   \end{bmatrix}. 
   \label{eq:28}
\end{equation}
\end{theorem}
\begin{figure}[H]
    \centering
    \begin{minipage}{0.5\textwidth}
        \centering
        \includegraphics[width=\textwidth,, keepaspectratio]{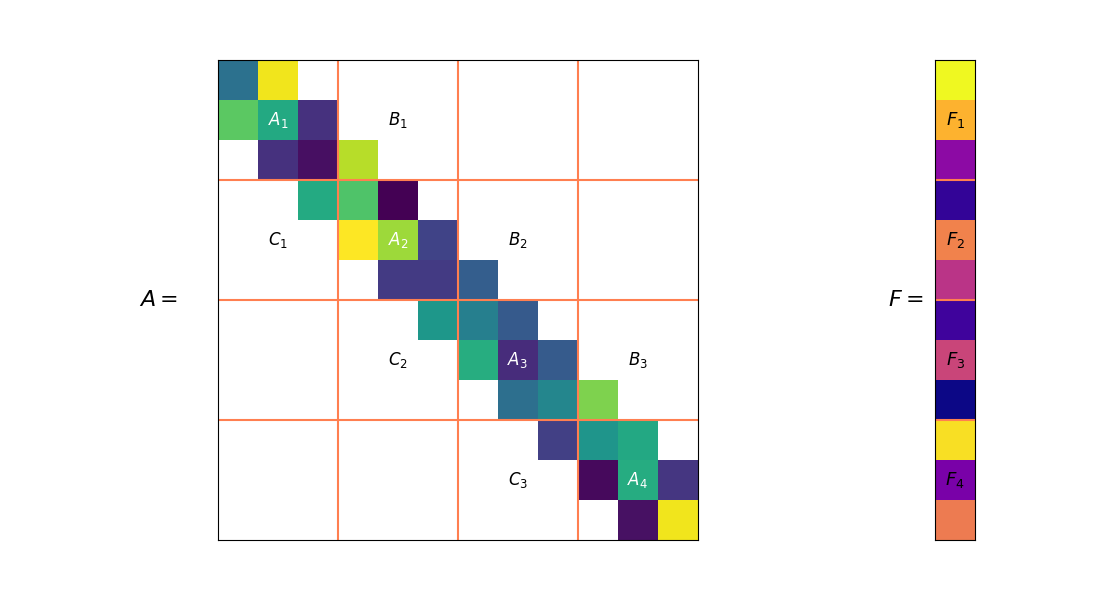}
    \end{minipage}%
    \hfill
    \begin{minipage}{0.5\textwidth}
        \centering
        \includegraphics[width=\textwidth,keepaspectratio]{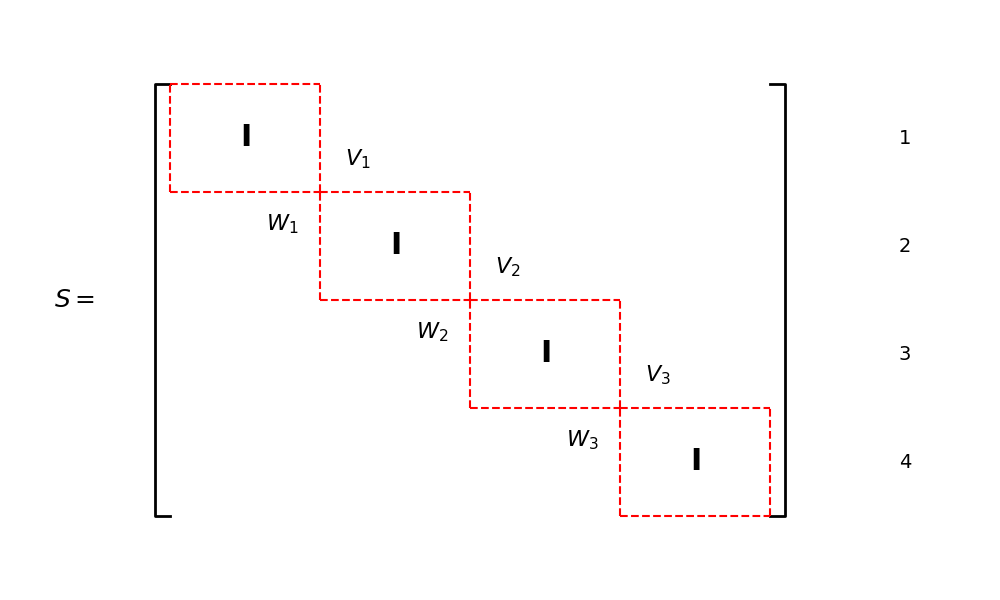}
    \end{minipage}
    \caption{(a)Partitioned matrix $\mathbf{A}$ and right-hand side $\mathbf{F}$ ($r = 4$).
 (b)$S$ matrix structure with $r=4$ partition}
    \label{fig:comparison}
\end{figure}
\begin{proof}
The result can be proven easily once it is established that a monotone matrix leaves inequalities invariant under the decomposition described in the theorem. Importantly, such a construction can be implemented because the monotone matrix preserves the inequality by virtue of Theorem 3.1. Hence, without loss of generality, the original problem can be transformed into the form of ~\eqref{eq:32}.
\end{proof}
Importantly, the solution of the LCP ~\eqref{eq:31} can be decoupled on a parallel computer and equivalently solved by constructing two sets of minimization problems, namely
\begin{itemize}
 \item $P_1$: The first one and its equivalent form are defined as 
\[\min_{x\in \sum_{i=1}^{r}\mathcal{K}_{i,in} \oplus \mathcal{K}_{i,out} \subset \mathcal{K}}  \mathcal{F}_{1}= \dfrac{1}{2}x^{T}\tilde{S}x-\tilde{G}^{T}x \]
$$\Updownarrow$$
   	\begin{equation}
	   	\begin{cases}
   		\tilde{S}\tilde{X} \geq \tilde{G} \\
   		\tilde{X} \geq \mathbf{0} \\
   		(\tilde{S}\tilde{X}-\tilde{G}). \tilde{X}=0
   	\end{cases}
   	\label{eq:29}
   \end{equation}
where $\tilde{S},\tilde{G}$ are constructed as from Eqn~\eqref{eq:32} by considering top $\beta$ rows and bottom $\beta$ rows from each partioned block of $S,G$ matrices with total equation size $n_{rd}=2*r*\beta$.
 \item $P_2$: The second one along with equivalent formulation defined as
\[\min_{x\in \sum_{i=1}^{r}\mathcal{K}_{i,mid} \subset \mathcal{K}} \mathcal{F}_{2}= \dfrac{1}{2}x^{T}\tilde{I}x-\tilde{G^{*}}^{T}x \]
$$\Updownarrow$$
   	\begin{equation}
	   	\begin{cases}
   		\tilde{I_{*}}X^{*} \geq G^{*} \\
   		X^{*} \geq \mathbf{0} \\
   		(\tilde{I_{*}}X^{*}-G^{*}). X^{*}=0
   	\end{cases}
   	\label{eq:30}
   \end{equation} 
\end{itemize}
The solution to the decoupled minimization problems in Equations \eqref{eq:29} and \eqref{eq:30} is addressed using both direct and iterative approaches. This section first delineates the algorithm, followed by a formal mathematical justification of its convergence to the exact solution.
\subsubsection{Direct projection type iterative parallel approach}
We consider the decoupled linear complementarity systems given in \eqref{eq:29} and \eqref{eq:30}, which are analogous to the preprocessing and postprocessing stages of the SPIKE algorithm introduced by Sameh and Polizzi \cite{Spring2020SPIKE,polizzi2006parallel}.
The reduced linear complementarity system (RLCS) is derived by considering the $m$ top and bottom rows from each $r$-partitioned  block of the SPIKE-type matrix shown in Fig. 2. The solution of these decoupled problems can be divided into two stages.
To solve this reduced $\text{LCS}$, we apply an $\text{LU}$ or $\text{UL}$ decomposition combined with a projection step in the final stage of the factorization. 
Once the reduced system is resolved, the remaining solution components associated with Equation \eqref{eq:30} can be decoupled straightforwardly. By applying a similar projection in the final stage, we completely recover the full solution vector.
To establish the unique solvability of the original $\text{LCS}$ problem, we prove the existence of the solution under the assumption that the original matrix $A$ is an $M$-matrix. 
Consequently, each partitioned diagonal block is likewise an $M$-matrix, meaning that the matrix $D$ in the aforementioned decomposition is an $M$-matrix as well. 
Therefore, to ensure a unique solution to the original problem, it is sufficient to establish the following theorem, which asserts the desired uniqueness result.
\begin{lemma}
 Let $A$ be a $M$-matrix and let $A=LU$ and $A=UL$ two decompositions of $A$. Then $L^{-1}\ge 0,U^{-1}\ge 0$.
\end{lemma}
\begin{theorem}
Let $A$ a be $M$-matrix then partioned LCS equation \eqref{eq:29} and \eqref{eq:30} has a unique solution. 
\end{theorem}
\begin{proof}
 Since the original matrix $A$ is an $M$-matrix, each partitioned diagonal block $A_i$ for all $i=1,2,\dots,r$ is likewise an $M$-matrix. 
Consequently, the decoupled block diagonal matrix $D$ remains an $M$-matrix as it inherits the structural properties of $A$. 
Notably, the off-diagonal entries of the SPIKE matrix $S$ are strictly non-positive with moduli strictly bounded by one, which guarantees that $S$ is also an $M$-matrix. 
Because $S$ is an $M$-matrix, extracting its top $m$ and bottom $m$ rows preserves this structure, implying that the reduced matrix $\tilde{S}$ is an $M$-matrix as well. 
The existence of the solution to the linear complementarity systems (LCS) \eqref{eq:29} and \eqref{eq:30} hinges entirely on finding a unique solution to the reduced linear complementarity problem $\text{LCP}(q,\tilde{S})$. 
Once the reduced solution is computed, the remaining component in \eqref{eq:30} yields a unique solution via the trivial identity problem $\text{LCP}(q,\tilde{I}_{*})$, which inherently satisfies all $M$-matrix properties. 
Therefore, to establish the main result, it is sufficient to prove that the reduced LCS equation \eqref{eq:29} possesses a unique solution. 
By employing a standard $\text{LU}$ decomposition combined with a projection onto the convex set, we successfully obtain the unique solution to $\text{LCP}(q,\tilde{S})$.
\end{proof}
\begin{remark}
 In essence, the reduced complementarity system preserves a banded structure identical to that of the original problem. Due to its intrinsic recursive nature, this complemeatrity system can be efficiently resolved by executing a recursive routine within the computational algorithm. The procedure of implementation is defined below.  
\end{remark}
\begin{algorithm}[{\bf Recursive variants of the algorithm}]\leavevmode\par
\begin{algorithmic}
\Procedure{Recursive Projected SPIKE}{$i_{\text{in}}, i_{\text{out}}, A, x$}
    \State $n_{\text{loc}} \gets i_{\text{out}} - i_{\text{in}} + 1$
    \If{$n_{\text{loc}} \le \text{tol}$}
     \State Solve LCS $$(R \cdot y - f_{r} ) \ge  0, \quad y \ge 0, \quad (R \cdot y- f_{r} )\perp  y=0 $$
        \State \textbf{Return}
    \EndIf
    \State $m_1 \gets \lfloor n_{\text{loc}} / 2 \rfloor$; \quad $i_{\text{mid}} \gets i_{\text{in}} + m_1 - 1$
    \State \Call{Recursive Projected SPIKE}{$i_{\text{in}}, i_{\text{mid}}, A, x$} 
    
    \State \Call{Recursive Projected SPIKE}{$i_{\text{mid}}+1, i_{\text{out}}, A, x$} 
    
    \State Compute Spike Arrays $V = A_{\text{left}}^{-1}B_1$ and $W = A_{\text{right}}^{-1}B_2$
    
    \State Solve recollected RLCS $$(R \cdot y - f_r ) \ge  0, \quad y \ge 0, \quad (R \cdot y - f_r )\perp  y=0 $$ 
                      Compute interface coupling coefficients vector $y$
    \State $x(i_{\text{in}}:i_{\text{mid}}) \gets \max (x(i_{\text{in}}:i_{\text{mid}}) - V \cdot y_2,0)$
    
    \State $x(i_{\text{mid}}+1:i_{\text{out}}) \gets \max (x(i_{\text{mid}}+1:i_{\text{out}}) - W \cdot y_1,0)$
\EndProcedure
\end{algorithmic}
\end{algorithm}
\subsubsection{Iterative Parallel Appraoch}
\begin{algorithm}[{\bf Type I}]\leavevmode\par
Choose $x^{0} \in V$ and $\gamma_{1},\gamma_{2}>0$\\
For $k=0,1,...$ $\quad  \quad $ (!Until $\nabla \mathcal{F}_{1} \textit{ and } \nabla \mathcal{F}_{2} = 0$).\\
Decompose the domain as 
\[\mathcal{K}=\sum_{i=1}^{r}\mathcal{K}_{i}\]
Consider a space of decomposition into two parts as \\
$$\mathcal{K}= \mathcal{K}_{1}^{*}  \oplus  \mathcal{K}_{2}^{*},$$
where $\mathcal{K}_{1}^{*}=\sum_{i=1}^{r}\mathcal{K}_{i,in}\oplus \mathcal{K}_{i,out}$ and
$\mathcal{K}_{2}^{*}=\sum_{i=1}^{r}\mathcal{K}_{i,mid}$.\\
Compute $y_{1}^{k} \in \mathcal{K}_{1}$ and $y_{2}^{k} \in \mathcal{K}_{2}$ such that 
$$\mathcal{F}_{1}(x^{k}+ \mathcal{P}_{in \oplus out} y_{1}^{k})=\min_{y_{1} \in \mathcal{K}_{1} }\mathcal{F}_{1}(x^{k}+ \mathcal{P}_{i,in \oplus out} y_{1})$$
$$\mathcal{F}_{2}(x^{k}+ \mathcal{P}_{in \oplus out} y_{2}^{k})=\min_{y_{2} \in \mathcal{K}_{2} }\mathcal{F}_{2}(x^{k}+ \mathcal{P}_{mid} y_{2})$$
$$x^{1,k}=x^{k}+\mathcal{P}_{in \oplus out}y_{1}^{k}$$
$$x^{2,k}=x^{k}+ \mathcal{P}_{mid} y_{2}^{k}$$
Determine $\delta_{1}^{k},\delta_{2}^{k}$ from the below update
\begin{align}
 x^{k+1}=x^{k}+\delta_{1}^{k}\mathcal{P}_{in \oplus out} y_{1}^{k} +\delta_{1}^{k}\mathcal{P}_{mid} y_{2}^{k}
 \label{eq:31}
\end{align}
such that 
\begin{align}
 \mathcal{F}(x^{k+1}) \le \gamma_{1}\mathcal{F}_{1}(x^{1,k})+ \gamma_{2}\mathcal{F}_{2}(x^{2,k})
 \label{eq:32}
\end{align}
End
\end{algorithm}
The above last inequality \eqref{eq:36} can be ensure by considering suitable parameters $\delta_{1}^{k},\delta_{2}^{k}$ in the either form
\begin{itemize}
 \item Determine $t$ such that 
 \[\mathcal{F}(x^{t,k})=\min (\mathcal{F}_{1}(x^{k}),\mathcal{F}_{2}(x^{k})) \]
 \item Using the convexity property of $\mathcal{F}$ 
$$x^{k+1}=x^{k}+\delta_{1}^{k}\mathcal{P}_{in \oplus out} y_{1}^{k} +\delta_{1}^{k}\mathcal{P}_{mid} y_{2}^{k}
=x^{k}+\gamma_{1}^{k}\mathcal{P}_{in \oplus out} y_{1}^{k} +\gamma_{2}^{k}\mathcal{P}_{mid} y_{2}^{k} $$
with $\delta_{1}=\gamma_{1}$ and $\delta_{2}=\gamma_{2}$.
\end{itemize}
\begin{algorithm}[{\bf Type II}]\leavevmode\par
Choose $x^{0} \in \mathcal{K}$ \\
For \quad $k=0,1,..\quad \quad \quad \quad \quad $ \ \ \  (!Until $\nabla \mathcal{F}_{1} , \nabla \mathcal{F}_{2} =0$) \\
Choose a space decomposition $$\mathcal{K}=\sum_{i=1}^{r} \mathcal{K}_{i}$$
Consider a space of decomposition into two parts as
$$\mathcal{K}= \mathcal{K}_{1}^{*}  \oplus  \mathcal{K}_{2}^{*},$$
where $\mathcal{K}_{1}^{*}=\sum_{i=1}^{r}\mathcal{K}_{i,in}\oplus \mathcal{K}_{i,out}$ and
$\mathcal{K}_{2}^{*}=\sum_{i=1}^{r}\mathcal{K}_{i,mid}$.
Compute $y^{k}_{1}\in \mathcal{K}_{1}^{*}$ and $y^{k}_{2} \in \mathcal{K}_{2}^{*}$ such that
$$\mathcal{F}_{1}(x^{k}+ \mathcal{P}_{in \oplus out} y_{1}^{k})=\min_{y_{1} \in \mathcal{K}_{1} }\mathcal{F}_{1}(x^{k}+ \mathcal{P}_{i,in \oplus out} y_{1})$$
$$\mathcal{F}_{2}(x^{k}+ \omega \mathcal{P}_{in \oplus out} y_{1}^{k}+\mathcal{P}_{mid} y_{2}^{k})=\min_{y_{2} \in \mathcal{K}_{2} }\mathcal{F}_{2}(x^{k}+ \omega \mathcal{P}_{in \oplus out} y_{1}^{k} + \mathcal{P}_{mid} y_{2})$$
$$x^{k+1}=x^{k}+\omega (\mathcal{P}_{in \oplus out} y_{1}^{k} + \mathcal{P}_{mid} y_{2}^{k})$$
End
\end{algorithm}
\section{General convergence theory for distributive implementation}
In this section, we first state and derive the unconstraint optimization formulation defined in Section 1. We will try to establish some convergence for linear type problem.
\subsection{Parallel distributive analysis for general elliptic problem : smooth optimization}
Additionally, we assume the functionals $\mathcal{F}_{1}$ and $\mathcal{F}_{2}$ are $K_1$-smooth and $K_2$-smooth, respectively. By the standard Descent Lemma, this property is equivalent to satisfying the following quadratic upper bounds for all $\mathbf{x}, \mathbf{y} \in \mathcal{K}$:
\[
\mathcal{F}_1(\mathbf{y})-  \mathcal{F}_1(\mathbf{x})\le  \langle \nabla \mathcal{F}_1(\mathbf{x}), \mathbf{y}-\mathbf{x} \rangle + \frac{K_1}{2} \|\mathbf{y}-\mathbf{x}\|^2
\]
\[
\mathcal{F}_2(\mathbf{y})-\mathcal{F}_2(\mathbf{x}) \le   \langle \nabla \mathcal{F}_2(\mathbf{x}), \mathbf{y}-\mathbf{x} \rangle + \frac{K_2}{2} \|\mathbf{y}-\mathbf{x}\|^2
\]
In this section, we assume that the functionals $\mathcal{F}_{1}$ and $\mathcal{F}_{2}$ are $K_1$ and $K_2$-smooth. By utilizing this standard Lipschitz-continuous gradient property, we invoke the classic descent lemma to establish the convergence of Algorithm 1.
We now state and prove the convergence theorem for Algorithm 1.
\begin{theorem}
Let $\mathcal{F}_{1}$ and $\mathcal{F}_{2}$ be functionals that are bounded from below and possess Lipschitz-continuous gradients. Suppose the space $\mathcal{K}$ is decomposed as $\mathcal{K}=\sum_{i=1}^{r}\mathcal{K}_{i}$, where each subspace is partitioned into $\mathcal{K}_{i}=\mathcal{K}_{i,\text{in}}\oplus \mathcal{K}_{i,\text{mid}} \oplus \mathcal{K}_{i,\text{out}}$. Alternatively, let the space be represented as $\mathcal{K}=\mathcal{K}_{1}^{*}\oplus \mathcal{K}_{2}^{*}$, where $\mathcal{K}_{1}^{*}=\sum_{i=1}^{r}\mathcal{K}_{i,\text{in}}\oplus \mathcal{K}_{i,\text{out}}$ and $\mathcal{K}_{2}^{*}=\sum_{i=1}^{r}\mathcal{K}_{i,\text{mid}}$. Furthermore, let $\mathcal{R}_{\text{in} \oplus \text{out}}:\mathcal{K} \to \mathcal{K}_{1}^{*}$ and $\mathcal{R}_{\text{mid}}:\mathcal{K} \to \mathcal{K}_{2}^{*}$ be surjective restriction operators, and let $\mathcal{P}_{\text{in} \oplus \text{out}}: \mathcal{K}_{1}^{*} \to \mathcal{K}$ and $\mathcal{P}_{\text{mid}}: \mathcal{K}_{2}^{*} \to \mathcal{K}$ be linear, injective embedding operators satisfying $\mathcal{P}_{\text{in} \oplus \text{out}}^{T}=\mathcal{R}_{\text{in} \oplus \text{out}}$ and $\mathcal{P}_{\text{mid}}^{T}=\mathcal{R}_{\text{mid}}$. Assume that the space decomposition satisfies the following stability condition for some constant $c>0$:
\[
\|\mathcal{R}_{\text{in} \oplus \text{out}} \mathbf{x}\|^{2}+\|\mathcal{R}_{\text{mid}} \mathbf{x} \|^{2} \ge c\|\mathbf{x}\|^{2} \quad \forall \mathbf{x} \in \mathcal{K}
\]
Then, every limit point of the sequence 
\[
\{\mathbf{x}^k\} := \begin{Bmatrix}
\mathbf{x}^{k}_{1} \\
\mathbf{x}^{k}_{2}
\end{Bmatrix}
\]
generated by Algorithm 1 is a solution to the minimization problem, satisfying:
\[
\lim_{k \to \infty}\nabla \mathcal{F}(\mathbf{x}^{k}) = \begin{Bmatrix}
\lim_{k \to \infty}\nabla \mathcal{F}_{1}(\mathbf{x}^{k}_{1}) \\
\lim_{k \to \infty}\nabla \mathcal{F}_{2}(\mathbf{x}^{k}_{2})
\end{Bmatrix} = \begin{Bmatrix}
\mathbf{0} \\
\mathbf{0}
\end{Bmatrix}
\]
\end{theorem}

\begin{proof}
 We observe that 
 \[\nabla \mathcal{F}_{1}(y_{1})=\nabla \mathcal{F}(x^{k}_{1}+\mathcal{P}_{in \oplus out}y_{1})\mathcal{P}_{in \oplus out}\]
 \[\nabla \mathcal{F}_{2}(y_{2})=\nabla \mathcal{F}(x^{k}_{2}+\mathcal{P}_{mid}y_{2})\mathcal{P}_{mid}\]
 As we assume that $\mathcal{F}_{1}$ $\mathcal{F}_{2}$ are Lipschitz continuous gradients with constants $K_1$ and $K_2$ we get that
 \[ ||\nabla \mathcal{F}(y_{1})-\nabla \mathcal{F}(y_{1})|| \le ||\nabla \mathcal{F}_{1}(y_{1})-\nabla \mathcal{F}_{1}(y_{1})||+ ||\nabla \mathcal{F}_{2}(y_{1})-\nabla \mathcal{F}_{2}(y_{1})|| \]
 \[= ||\nabla \mathcal{F}(x^{k}_{1}+\mathcal{P}_{in \oplus out}y_{1})\mathcal{P}_{in \oplus out}-\nabla \mathcal{F}(x^{k}_{1}+\mathcal{P}_{in \oplus out}\bar{y}_{1})\mathcal{P}_{in \oplus out}||+||\nabla \mathcal{F}(x^{k}_{2}+\mathcal{P}_{mid}y_{1})\mathcal{P}_{mid}-\nabla \mathcal{F}(x^{k}_{2}+\mathcal{P}_{mid}\bar{y}_{1})\mathcal{P}_{mid}||\]
 \[\le ||\mathcal{P}_{in \oplus out}||.K_{1}||\mathcal{P}_{in \oplus out}(y_{1}-\bar{y}_{1})||+ ||\mathcal{P}_{mid}||.K_{2}||\mathcal{P}_{mid}(y_{1}-\bar{y}_{1})||\]
 \[||\mathcal{P}_{in \oplus out}||^{2}. K_{1}||(y_{1}-\bar{y}_{1})||+||\mathcal{P}_{mid}||^{2}. K_{2}||(y_{1}-\bar{y}_{1})||\]
 \[\le K.(||\mathcal{P}_{in \oplus out}||^{2}+||\mathcal{P}_{mid}||^{2})||(y_{1}-\bar{y}_{1})||,\]
 here $K=\max(K_{1},K_{2})$.\\
which shows that $\nabla \mathcal{F}$ is lipschitz continuous.
Let $z_{1}^{k}=-(1/K).\nabla \mathcal{F}_{1}(0)^{T}$ and $z_{2}^{k}=-(1/K).\nabla \mathcal{F}_{2}(0)^{T}$. Then by using quardratic Bound lemma we able to get
\[\nabla \mathcal{F}_{1}(0)-\nabla \mathcal{F}_{1}(z_{1}^{k}) \ge \frac{1}{2K}||\nabla \mathcal{F}_{1}(0)^{T}||\]
\[\nabla \mathcal{F}_{2}(0)-\nabla \mathcal{F}_{2}(z_{2}^{k}) \ge \frac{1}{2K}||\nabla \mathcal{F}_{2}(0)^{T}||\]
As we know from the minimization property,
$\mathcal{F}_{1}(z_{1}^{k}) \ge \mathcal{F}_{1}(y_{1}^{k})$ and 
$\mathcal{F}_{2}(z_{2}^{k}) \ge \mathcal{F}_{2}(y_{2}^{k})$ implies 
\[\mathcal{F}(x^{k})-\mathcal{F}(x^{1,k}) \ge  \mathcal{F}_{1}(0)-\mathcal{F}_{1}(y_{1}^{k}) \ge \frac{1}{2K}||\nabla \mathcal{F}_{1}(0)^{T}||\]
and
\[\mathcal{F}(x^{k})-\mathcal{F}(x^{2,k}) \ge  \mathcal{F}_{2}(0)-\mathcal{F}_{2}(y_{2}^{k}) \ge \frac{1}{2K}||\nabla \mathcal{F}_{2}(0)^{T}||\]
which implies that 
\[\mathcal{F}(x^{k})-\mathcal{F}(x^{1,k}) \ge \frac{1}{2K}||(\nabla \mathcal{F}(x^{k})\mathcal{P}_{in \oplus out})^{T}||^2  = \frac{1}{2K}||\mathcal{R}_{in \oplus out}\nabla \mathcal{F}(x^{k})^{T}||^2\]
and
\[\mathcal{F}(x^{k})-\mathcal{F}(x^{2,k}) \ge \frac{1}{2K}||(\nabla \mathcal{F}(x^{k})\mathcal{P}_{mid})^{T}||^2 = \frac{1}{2K}||\mathcal{R}_{mid}\nabla \mathcal{F}(x^{k})^{T}||^2\] 
Now we multiply both side by weight factor $\beta_{i} \in (0,1)$ and using above argument able to get the following
\[\mathcal{F}(x^{k})-(\beta_{1}\mathcal{F}(x^{1,k})+\beta_{2}\mathcal{F}(x^{2,k}))\ge \frac{c\beta}{2K}||\nabla \mathcal{F}(x^{k})^{T}||^2,\]
where $\beta=\min(\beta_{1},\beta_{2})$.\\
Now using the linear convergence theory every limit point of $\{x^{k}\}$ is convergent, that is
$$\lim_{n \to \infty}\nabla \mathcal{F}(x^{k}):=\begin{Bmatrix}
\lim_{n \to \infty}\nabla \mathcal{F}_{1}(x^{k}_{1})\\
\lim_{n \to \infty}\nabla \mathcal{F}_{2}(x^{k}_{2})
\end{Bmatrix}
=\begin{Bmatrix}
0\\
0
\end{Bmatrix}.$$
\end{proof}
\begin{lemma}
 Let $\mathcal{F}_{1}: \mathcal{K}_{1}^{*} \to \mathbb{R}$ and $\mathcal{F}_{2}: \mathcal{K}_{2}^{*} \to \mathbb{R}$ are continuously differential functions and let 
 $\{x^{k}_{1}\} \subseteq \mathcal{K}_{1}^{*}$ and $\{x^{k}_{2}\} \subseteq \mathcal{K}_{2}^{*} $. If $\mathcal{F}_{1}$ and $\mathcal{F}_{2}$ are bound from below and 
 \[\mathcal{F}_{1}(x^k)-\mathcal{F}_{1}(x^{k+1}) \ge \alpha_{1}  ||\nabla \mathcal{F}_{1}(x^k)||^2 \]
 and
 \[\mathcal{F}_{2}(x^k)-\mathcal{F}_{2}(x^{k+1}) \ge \alpha_{2}  ||\nabla \mathcal{F}_{2}(x^k)||^2 \]
\end{lemma}
\begin{theorem}
 Let the functionals $\mathcal{F}_{1}$ and $\mathcal{F}_{2}$ have Lipschitz-continuous gradients, be bounded from below as stated in Lemma 4.2, and be strongly convex with constants $C_{1}$ and $C_{2}$, respectively. Then, the sequence of iterates $\{\mathbf{x}^{k}\}$ converges to the unique minimizer $\mathbf{x}^*$ of $\mathcal{F}$ at a linear rate, satisfying 
 \[||\mathbf{x^{k}}-\mathbf{x^*}||\le \Big(\frac{2}{C_{1}}(\mathcal{F}_{1}(x^{0}_{1})-\mathcal{F}_{1}(x_{1}^{k})) \Big)^{1/2}\Big( 1-\frac{c_{1}\beta_{1}C_{1}^2}{K^2}\Big)^{k/2} 
 +\Big(\frac{2}{C_{2}}(\mathcal{F}_{2}(x^{0}_{2})-\mathcal{F}_{2}(x_{2}^{k})) \Big)^{1/2}\Big( 1-\frac{c_{2}\beta_{2}C_{2}^2}{K^2}\Big)^{k/2}.\]
\end{theorem}
\begin{proof}
First part of the proof is similar to the result presented by \cite{ferris1994parallel}.
We now derive the linear root rate of convergence for $\{\mathbf{x}^{i}\}$. Utilizing the strong convexity of $\mathcal{F}_{1}$ and $\mathcal{F}_{2}$, we obtain the following relations
 \[||\nabla \mathcal{F}_{1}(x_{1})||||x^{i}_{1}-\bar{x}_{1}||=||\nabla \mathcal{F}_{1}(x^{i}_{1})-\nabla \mathcal{F}_{1}(\bar{x}_{1})||x^{i}_{1}-\bar{x}_{1}|| \ge \nabla \mathcal{F}_{1}(x^{i}_{1})-\nabla \mathcal{F}_{1}(\bar{x}_{1})(x^{i}_{1}-\bar{x}_{1}) \ge k_{1}||x^{i}_{1}-\bar{x}_{1}||^2 \]
 \[||\nabla \mathcal{F}_{2}(x^{i}_{2})||||x^{i}_{2}-\bar{x}_{2}||=||\nabla \mathcal{F}_{1}(x^{i}_{2})-\nabla \mathcal{F}_{2}(\bar{x}_{2})||x^{i}_{2}-\bar{x}_{2}|| \ge \nabla \mathcal{F}_{2}(x^{i}_{2})-\nabla \mathcal{F}_{2}(\bar{x}_{1})(x^{i}_{2}-\bar{x}_{2}) \ge k_{2}||x^{i}_{2}-\bar{x}_{2}||^2\]
 This implies 
 \[\mathcal{F}_{1}(x^{i}_{1})-\mathcal{F}_{1}(x^{i+1}_{1}) \ge \alpha_{1}k_{1}^{2}||x^{i}_{1}-\bar{x}_{1}||^2\]
 \[\mathcal{F}_{2}(x^{i}_{2})-\mathcal{F}_{2}(x^{i+1}_{2}) \ge \alpha_{2}k_{2}^{2}||x^{i}_{2}-\bar{x}_{2}||^2\]
 Now using quadratic bound lemma we get
 \[(\mathcal{F}_{1}(x^{i}_{1})-\mathcal{F}_{1}(x^{i+1}_{1}) \ge \frac{2\alpha_{1}k_{1}^{2}}{K_{1}} \]
 \[(\mathcal{F}_{2}(x^{i}_{2})-\mathcal{F}_{2}(x^{i+1}_{2}) \ge \frac{2\alpha_{2}k_{2}^{2}}{K_{2}}\]
 which can be rewrite as
 \[\Big(1-\frac{2\alpha_{1}k_{1}^{2}}{K_{1}}\Big)(\mathcal{F}_{1}(x^{i}_{1})-\mathcal{F}_{1}(\bar{x}_{1})) \ge \mathcal{F}_{1}(x^{i+1}_{1})-\mathcal{F}_{1}(\bar{x}_{1})\]
 \[\Big(1-\frac{2\alpha_{2}k_{2}^{2}}{K_{2}}\Big)(\mathcal{F}_{2}(x^{i}_{2})-\mathcal{F}_{2}(\bar{x}_{2})) \ge \mathcal{F}_{2}(x^{i+1}_{2})-\mathcal{F}_{2}(\bar{x}_{2})\]
 Using the same argument iteratively we able to get
 \[(\mathcal{F}_{1}(x^{i}_{1})-\mathcal{F}_{1}(\bar{x}_{1})) \le \Big(1-\frac{2\alpha_{1}k_{1}^{2}}{K_{1}}\Big) (\mathcal{F}_{1}(x^{0}_{1})-\mathcal{F}_{1}(\bar{x}_{1}))\]
 \[(\mathcal{F}_{2}(x^{i}_{2})-\mathcal{F}_{2}(\bar{x}_{2})) \le \Big(1-\frac{2\alpha_{2}k_{2}^{2}}{K_{2}}\Big)(\mathcal{F}_{2}(x^{0}_{2})-\mathcal{F}_{2}(\bar{x}_{2}))\]
 This implies 
 \[(\mathcal{F}_{1}(x^{i}_{1})-\mathcal{F}_{1}(\bar{x}_{1}))+(\mathcal{F}_{2}(x^{i}_{2})-\mathcal{F}_{2}(\bar{x}_{2})) \le \Big(1-\frac{2\alpha_{1}k_{1}^{2}}{K_{1}}\Big) (\mathcal{F}_{1}(x^{0}_{1})-\mathcal{F}_{1}(\bar{x}_{1}))+\Big(1-\frac{2\alpha_{2}k_{2}^{2}}{K_{2}}\Big)(\mathcal{F}_{2}(x^{0}_{2})-\mathcal{F}_{2}(\bar{x}_{2}))\]
 Again from strong convexity property we have that 
\[(\mathcal{F}_{1}(x^{i}_{1})-\mathcal{F}_{1}(\bar{x}_{1})) \ge \frac{k_{1}}{2}||x^{i}_{1}-\bar{x}_{1}||^{2}+ \nabla \mathcal{F}_{1} (\bar{x}_{1})(x^{i}_{1}-\bar{x}_{1}) \]
\[(\mathcal{F}_{2}(x^{i}_{2})-\mathcal{F}_{2}(\bar{x}_{2})) \ge \frac{k_{2}}{2}||x^{i}_{2}-\bar{x}_{2}||^{2} + \nabla \mathcal{F}_{2}(\bar{x}_{2})(x^{i}_{2}-\bar{x}_{2})\]
 or 
 \[||x^{i}_{1}-\bar{x}_{1}|| \le \Big(\frac{2}{k_{1}}(\mathcal{F}_{1}(x^{i}_{1})-\mathcal{F}_{1}(\bar{x}_{1})) \Big)^{1/2}\]
 \[||x^{i}_{2}-\bar{x}_{2}|| \le \Big(\frac{2}{k_{2}}(\mathcal{F}_{2}(x^{i}_{2})-\mathcal{F}_{2}(\bar{x}_{2})) \Big)^{1/2}\]
This implies 
 \[(\mathcal{F}_{1}(x^{i}_{1})-\mathcal{F}_{1}(\bar{x}_{1}))+(\mathcal{F}_{2}(x^{i}_{2})-\mathcal{F}_{2}(\bar{x}_{2})) \le \Big(1-\frac{2\alpha_{1}k_{1}^{2}}{K_{1}}\Big) (\mathcal{F}_{1}(x^{0}_{1})-\mathcal{F}_{1}(\bar{x}_{1}))+\Big(1-\frac{2\alpha_{2}k_{2}^{2}}{K_{2}}\Big)(\mathcal{F}_{2}(x^{0}_{2})-\mathcal{F}_{2}(\bar{x}_{2}))\] 
 Now using above we able to achieve 
 \[||\mathbf{x^{k}}-\mathbf{x^*}|| \le ||x^{i}_{1}-\bar{x}_{1}||+||x^{i}_{2}-\bar{x}_{2}|| \le \]
 \[ \Big(\frac{2}{C_{1}}(\mathcal{F}_{1}(x^{0}_{1})-\mathcal{F}_{1}(x_{1}^{k})) \Big)^{1/2}\Big( 1-\frac{c_{1}\beta_{1}C_{1}^2}{K^2}\Big)^{k/2} 
 +\Big(\frac{2}{C_{2}}(\mathcal{F}_{2}(x^{0}_{2})-\mathcal{F}_{2}(x_{2}^{k})) \Big)^{1/2}\Big( 1-\frac{c_{2}\beta_{2}C_{2}^2}{K^2}\Big)^{k/2}\]
\end{proof}
In practice, solving the local minimization problem exactly is often computationally unfeasible. The below theorem give guarantee about the convergence of such solution.
\begin{theorem}
Let the functionals $\mathcal{F}_{1}$ and $\mathcal{F}_{2}$ have Lipschitz-continuous gradients and be bounded from below, as stated in Lemma 4.2. Suppose the space $\mathcal{K}$ is partitioned into two orthogonal subspaces, $\mathcal{K}=\mathcal{K}_{1}^{*} \oplus \mathcal{K}_{2}^{*}$, as defined in Theorem 4.1. Additionally, assume that the following holds for some constant $c > 0$:
\[
\|\mathcal{R}_{\text{in} \oplus \text{out}} \mathbf{x}\|^{2} + \|\mathcal{R}_{\text{mid}} \mathbf{x} \|^{2} \ge c\|\mathbf{x}\|^{2} \quad \forall \mathbf{x} \in \mathcal{K}
\]
Let $\alpha_{1}, \alpha_{2} > 0$, and assume that Algorithm 1 is relaxed to accept an inexact solution $\mathbf{x}^{k+1}$ for the local minimization problem whenever the following descent conditions hold:
\[
\mathcal{F}_{1}(\mathbf{x}^k) - \mathcal{F}_{1}(\mathbf{x}^{k+1}) \ge \alpha_{1} \|\mathcal{R}_{\text{in} \oplus \text{out}} \nabla \mathcal{F}_{1}(\mathbf{x}^k)\|^2
\]
\[
\mathcal{F}_{2}(\mathbf{x}^k) - \mathcal{F}_{2}(\mathbf{x}^{k+1}) \ge \alpha_{2} \|\mathcal{R}_{\text{mid}} \nabla \mathcal{F}_{2}(\mathbf{x}^k)\|^2
\]
Then, every limit point of the sequence $\{\mathbf{x}^{k}\}$ generated by this relaxed algorithm is stationary, that is,
\[
\lim_{k \to \infty} \nabla \mathcal{F}(\mathbf{x}^{k}) = \begin{Bmatrix}
\lim_{k \to \infty} \nabla \mathcal{F}_{1}(\mathbf{x}^{k}_{1}) \\
\lim_{k \to \infty} \nabla \mathcal{F}_{2}(\mathbf{x}^{k}_{2})
\end{Bmatrix} = \mathbf{0}
\]
Additionally, if $\mathcal{F}_{1}$ and $\mathcal{F}_{2}$ are strongly convex and their gradients are Lipschitz continuous, then $\lim_{k \rightarrow \infty}\mathbf{x}^{k} = \mathbf{x}^{*}$, which is the unique minimizer of $\mathcal{F}$.
\end{theorem}
\begin{corollary}
Let $\mathcal{F}_{1}$ and $\mathcal{F}_{2}$ be strongly convex functionals with bounded-below, Lipschitz continuous gradients, satisfying the space decomposition $\mathcal{K} = \mathcal{K}_{1}^{*} \oplus \mathcal{K}_{2}^{*}$ from Theorem~4.1. Furthermore, let the synchronization step in Algorithm~1 be defined by $\mathbf{x}^{k+1} = \mathbf{x}^{k} + \gamma_{1} \mathcal{P}_{\mathrm{in} \oplus \mathrm{out}} y_{1}^{k} + \gamma_{2} \mathcal{P}_{\mathrm{mid}} y_{2}^{k}$ for positive step-sizes $\gamma_{1}, \gamma_{2}$. Then, the sequence $\{\mathbf{x}^{k}\}$ converges to the unique minimizer $\mathbf{x}^{*}$ of $\mathcal{F}$ as $k \to \infty$ under either of the following conditions
\begin{align*}
 \gamma = \gamma_{1} + \gamma_{2} \le 1 \textit{ and }
 \gamma = \gamma_{1} + \gamma_{2} < 2 
\end{align*}
, provided the generated spike matrix $S$ is symmetric.
\end{corollary}
\begin{proof}
\end{proof}
Next, we apply the above framework to treat the obstacle problem as a nonsmooth optimization problem.
\subsection{Parallel distributive analysis for obstacle problem : nonsmooth optimization}
Before we describe the convergence analysis of Algorithm 1, we first present the definitions and preliminary lemmas required for our main analysis. we recall the definition of optimality conditions
\begin{definition}[Optimality Function]
Consider our original problem $\min_{x \in \mathcal{K}} \mathcal{F}(x)$ where $\mathcal{F}: \mathbb{R}^n \rightarrow \mathbb{R}$ and $\mathcal{K} \subseteq \mathbb{R}^n$. A nonconstant, lower semicontinuous mapping $\zeta: \mathcal{K} \rightarrow \mathbb{R}$ is termed an optimality function if it satisfies:
\begin{enumerate}
    \item $\zeta(x) \ge 0$ for all $x \in \mathcal{K}$.
    \item $\zeta(x) = 0$ whenever $x \in \arg\min_{y \in \mathcal{K}} \mathcal{F}(y)$.
\end{enumerate}
\end{definition}
A point $x \in \mathcal{K}$ is classified as stationary with respect to 
the optimality function $\zeta$ if and only if $\zeta(x) = 0$. In the 
unconstrained, continuously differentiable setting (where 
$\mathcal{K} = \mathbb{R}^{n}$ and $\mathcal{F} \in C^{1}(\mathbb{R}^{n})$), 
$\zeta(x)$ simplifies naturally to the standard gradient norm 
$\|\nabla\mathcal{F}(x)\|$. Conversely, when constraints dictate that 
$\mathcal{K}$ is a closed convex subset of~$\mathbb{R}^{n}$, the first-order 
necessary conditions yield the minimum principle optimality function
\begin{align}
\psi(x) := -\min_{y} \left\{ \langle \nabla \mathcal{F}(x), y-x \rangle 
\;\middle|\; y \in \mathcal{K}, \, \|y-x\|_{\infty} \le \alpha \right\}.
\label{eq:33}
\end{align}
\begin{lemma}
If $x$ is a stationary point for the optimality function $\psi$ given above equation (33), then $x$ solve the minimization problem of LCS formed in Section 1.
\end{lemma}
\begin{lemma}
Let $\mathcal{K} \subseteq \mathbb{R}^{n}$ be a closed convex set and $\mathcal{F} \in LC^{1}_{K}(\mathbb{R}^{n})$. 
For a given $x \in \mathcal{K}$, let the descent direction $d$ be defined by
\begin{align}
 d \in \arg\min \left\{ \langle \nabla \mathcal{F}(x), v \rangle \;\middle|\; v \in \mathcal{K}-x, \, \|v\|_{\infty} \le \alpha \right\},
 \label{eq:34}
\end{align}
and let $\kappa \in \{2^{-m}\}_{m=0}^{\infty}$ be the stepsize chosen via the Armijo condition
\begin{align}
 \mathcal{F}(x) - \mathcal{F}(x+\kappa d) \ge -\frac{\kappa}{2} \langle \nabla \mathcal{F}(x), d \rangle.
 \label{eq:35}
\end{align}
Then, the updated point $x_{\mathrm{new}} = x + \kappa d$ satisfies
\begin{align}
 \mathcal{F}(x) - \mathcal{F}(x_{\mathrm{new}}) \ge \vartheta(\psi(x)),
 \label{eq:36}
\end{align}
where $\psi(x)$ is the minimum principle optimality function defined in ~\eqref{eq:33}
, and $\vartheta(t) = \min \bigl\{ \frac{1}{2}t, \frac{1}{4\alpha^{2}K}t^{2} \bigr\}$ 
serves as a forcing function.
\end{lemma}
\begin{remark}
The above results are important for accelerating convergence to achieve the optimality condition during iterations. This condition also provides the key concept of treating the problem as a system of time-dependent ODEs, allowing it to be solved using incremental step sizes.
\end{remark}
To solve the obstacle problem, we define the distributed optimality function 
$\zeta: \mathcal{K} \to \mathbb{R}$ by
\begin{align}
 \zeta(x) = \zeta_{\mathrm{in} \oplus \mathrm{out}}(x) + \zeta_{\mathrm{mid}}(x),
 \label{eq:37}
\end{align}
where $\zeta_{\mathrm{in} \oplus \mathrm{out}}: \mathcal{K}_{1}^{*} \to \mathbb{R}_{\ge 0}$ 
and $\zeta_{\mathrm{mid}}: \mathcal{K}_{2}^{*} \to \mathbb{R}_{\ge 0}$ are 
nonconstant, lower semicontinuous functions satisfying:
\begin{itemize}
 \item $x_1 \in \arg\min_{y \in \mathcal{K}_1^*} \mathcal{F}_1(y) \iff \zeta_{\mathrm{in} \oplus \mathrm{out}}(x_1) = 0$.
 \item $x_2 \in \arg\min_{y \in \mathcal{K}_2^*} \mathcal{F}_2(y) \iff \zeta_{\mathrm{mid}}(x_2) = 0$.
\end{itemize}
A point $x \in \mathcal{K}$ is stationary if $\zeta(x) = 0$. Since both 
components are non-negative, any global minimizer $x^* \in \arg\min_{x \in \mathcal{K}} \mathcal{F}(x)$ 
necessarily satisfies $\zeta(x^*) = 0$.
\begin{theorem}[Convergence of the Obstacle Algorithm]
Let $\mathcal{F}_{1}, \mathcal{F}_{2}$ be strongly convex functionals with lower-bounded, Lipschitz continuous gradients, and let $\mathcal{K}_{1}^{*}, \mathcal{K}_{2}^{*} \subseteq \mathcal{V}$ be convex subsets. If the sequence $\{d^{i}\}$ is bounded, then the sequence $\{x^{i}\}$ generated by Algorithm 1 either converges to a stationary point $x^*$, or all of its accumulation points are stationary.
\end{theorem}
\section{Numerical Experiment}
We perform multiple numerical simulations to illustrate the effectiveness of the proposed algorithms. To evaluate their performance, we first conduct tests on obstacle problems utilizing a direct parallel computational approach that incorporates the simple projection methods discussed in Section~\ref{sec:projection}. For the one-dimensional case, employing the direct projected SPIKE algorithm yields the exact solution to the 1D obstacle problem. To solve two- and three-dimensional obstacle problems, we use a directional splitting approach where the 1D SPIKE algorithm acts as a preconditioner, thereby treating each directional subproblem as a sequence of 1D obstacle problems.
 To solve two- and three-dimensional obstacle problems, we employ a projected gradient descent framework governed by an Armijo backtracking line search approach, which dynamically calculates step parameters based on a sufficient decrease in the energy functional.
 We then extend these experiments for general obstacle problems in $\mathbb{R}^n,n=1,2,3$ with an parallel iterative algorithm mentioned in section 3. Finally, we assess the performance of the algorithms on practical image deblurring problems
 	\begin{table}[!htbp]
		\caption{Projected SPIKE result from 2, 4 and 8 processors and for matrix order$n \times n$ where $n=16,32,64,128,512,1024, 2048,4096,8192$ bandwidth $\beta=2$.}
		\begin{center}
			\begin{tabular}{ccccc}
				\hline
				Grid size $n \times n$ & No of processors & (CPU-time in Sec ave) & (CPU-time in Hours ave) \\
				\hline
				$16$ & $2$       & $1.2564105000000001\times 10^{-3}$ & $ 3.4900291666666671\times 10^{-7}$\\ 
				
				$16$ & $4$       & $1.4426392500000000\times 10^{-3}$ & $4.0073312500000002\times 10^{-7}$ \\
				
				\hline 
				$32$ & $2$       & $1.7317710000000000\times 10^{-3}$ & $ 4.8104749999999997\times 10^{-7}$\\ 
				
				$32$ & $4$       & $2.1030914999999998\times 10^{-3}$ & $   5.8419208333333333\times 10^{-7}$ \\
				
				$32$ & $8$       & $3.2978951250000003\times 10^{-3}$ & $ 9.1608197916666672\times 10^{-7}$\\ 
				\hline
				$64$ & $2$       & $6.3709365000000004\times 10^{-3}$ & $ 1.7697045833333335\times 10^{-6}$\\ 
				
				$64$ & $4$       & $ 4.4428387500000006\times 10^{-3}$ & $   1.2341218750000002\times 10^{-6}$ \\
				
				$64$ & $8$       & $7.8692155000000003\times 10^{-3}$ & $ 2.1858931944444447\times 10^{-6}$\\ 
				\hline
				$128$ & $2$       & $1.2333342000000001\times 10^{-2}$ & $ 3.4259283333333336\times 10^{-6}$\\ 
				
				$128$ & $4$       & $1.3762024749999999\times 10^{-2}$ & $   3.8227846527777777\times 10^{-6}$ \\
				
				$128$ & $8$       & $8.9567842500000005\times 10^{-3}$ & $ 2.4879956250000003\times 10^{-6}$\\ 
				\hline
				$256$ & $2$       & $4.2950179000000005\times 10^{-2}$ & $ 1.1930605277777780\times 10^{-5}$\\ 
				
				$256$ & $4$       & $4.2253353250000000\times 10^{-2}$ & $   1.1737042569444444\times 10^{-5}$ \\
				
				$256$ & $8$       & $5.2808312000000003\times 10^{-2}$ & $  1.4668975555555557\times 10^{-5}$\\ 
				\hline
				$512$ & $2$       & $0.23135659549999998$ & $6.4265720972222213\times 10^{-5}$\\ 
				
				$512$ & $4$       & $0.13779315225000000$ & $3.8275875625000000\times 10^{-5}$ \\
				
				$512$ & $8$       & $0.23980285312500002$ & $6.6611903645833335\times 10^{-5}$\\ 
				\hline
				$1024$ & $2$       & $0.94131520550000003$ & $2.6147644597222222
				\times 10^{-4}$\\ 
				
				$1024$ & $4$       & $0.31377938049999998$ & $8.7160939027777776       \times 10^{-5}$ \\
				
				$1024$ & $8$       & $0.26540373987499999$ & $7.3723261076388886       \times 10^{-5}$\\ 
				\hline
				$2048$ & $2$       & $7.7126638784999999$ & $2.1424066329166665
				\times 10^{-3}$\\ 
				
				$2048$ & $4$       & $1.6991023162500001$ & $4.7197286562500004        \times 10^{-4}$ \\
				
				$2048$ & $8$       & $1.1280106016250002$ & $3.1333627822916675        \times 10^{-4}$\\ 
				\hline
				$4096$ & $2$       & $147.86486831849999$ & $4.1073574532916664
				\times 10^{-2}$\\ 
				
				$4096$ & $4$       & $12.171439289000002$ & $3.3809553580555563        \times 10^{-3}$ \\
				
				$4096$ & $8$       & $7.4110701309999989$ & $2.0586305919444441        \times 10^{-3}$\\ 
				\hline
				$8192$ & $2$       & $1861.0646910930000$ & $0.51696241419250000$ \\
				
				$8192$ & $4$       & $171.62190361750001$ & $4.7672751004861111        \times 10^{-2}$ \\
				
				$8192$ & $8$       & $77.764998189499991$ & $2.1601388385972219        \times 10^{-2}$\\ 
				\hline
			\end{tabular}
		\end{center}
	\end{table}
%
%
\subsection{Numerical example of obstacle problem $\mathbb{R}^{n},n=1,2,3$ }
\begin{example}
Let $(\Omega, f, \phi) = \big((0,1), 0, \infty\big)$, and introduce the bilinear form  $a(\cdot,\cdot)$ given by
\[
a(u,v)=\int_{0}^{1} u'(x) v'(x) dx.
\]
We consider two alternative obstacle profile $\Psi_{ob}(x)$ here, namely a piecewise quadratic profile \ref{eq:psi_sine} and a smooth trigonometric profile \ref{eq:psi_exp}:
\begin{subequations}
\label{eq:psi_examples}
\begin{align}
\Psi_{ob}(x) &=
\begin{cases}
	50 \sin^2(2\pi x) & \text{for } 0 \le x \le 0.25, \\[4pt]
	50 \cos^2(2\pi (x - 0.25)) & \text{for } 0.25 < x \le 0.5, \\[4pt]
	\Psi_{ob}(1-x) & \text{for } 0.5 < x \le 1,
\end{cases} \label{eq:psi_sine}
\end{align}
\begin{align}
\Psi_{ob}(x) &=
\begin{cases}
	10(e^{4x} - 1) & \text{for } 0 \le x \le 0.25, \\[4pt]
	10(e - 1) + 20(x - 0.25)^2 & \text{for } 0.25 < x \le 0.5, \\[4pt]
	\Psi_{ob}(1-x) & \text{for } 0.5 < x \le 1,
\end{cases} \label{eq:psi_exp}
\end{align}
\end{subequations}
\end{example}
\begin{example}
Let $(\Omega, f, \phi_{ob}) = \big((0,1), 0, \infty\big)$, where the bilinear form $a(\cdot,\cdot)$ is given by
\[
a(u,v) = \int_{0}^{1} u'(x) v'(x) dx.
\]
We investigate the numerical behavior using two separate step-function configurations for the obstacle profile
\begin{enumerate}
    \item \textbf{Single Step Profile:}
    \begin{equation}
    \Psi_1(x) = \begin{cases} 1 & \text{if } 0.35 \le x \le 0.65, \\ 0 & \text{otherwise.} \end{cases}
    \label{eq:single_step}
    \end{equation}
    
    \item \textbf{Double Step Profile:}
    \begin{equation}
    \Psi_2(x) = \begin{cases} 0.8 & \text{if } 0.2 \le x \le 0.4, \\ 0.4 & \text{if } 0.6 \le x \le 0.8, \\ 0 & \text{otherwise.} \end{cases}
    \label{eq:double_step}
    \end{equation}
\end{enumerate}
\end{example}
%
%
%
%
 \begin{example}
Consider an obstacle problem formulated on the square domain $\Omega = [-2,2] \times [-2,2]$, with the obstacle function $\Psi(x)$ defined as in \cite{tran2015}
\end{example}
\begin{example}
To examine the convergence and behavior of our method in the presence of non-smooth obstacles, we consider a composite geometric obstacle function defined by in \cite{tran2015}
\end{example}
\begin{example}
 We first consider a smooth, continuous spherical dome obstacle nested at the core of the domain workspace. The geometric constraint profile $\phi_{\text{smooth}}(x,y,z)$ is defined analytically by:
\begin{equation}
\phi_{\text{smooth}}(x,y,z) = 
\begin{cases} 
\sqrt{R^2 - (x^2 + y^2 + z^2)} & \text{if } \sqrt{x^2 + y^2 + z^2} \le R \\ 
-1.0 & \text{otherwise}
\end{cases}
\end{equation}
where the radius parameter is locked to $R = 0.8$. 
\end{example}
\begin{example}
 To evaluate the solver under extreme geometric conditions where classical derivative evaluations collapse, we introduce a non-differentiable, sharp four-sided pyramid configuration characterized by a discontinuous subgradient singular peak and sharp crease intersections. The boundary profile $\phi_{\text{nonsmooth}}(x,y,z)$ is defined mathematically as:
\begin{equation}
\phi_{\text{nonsmooth}}(x,y,z) = \max\left( -0.4, \, 1.0 - |x| - |y| - |z| \right)
\end{equation}
To alleviate numerical grid locking across the sharp derivative ridge lines, a regularized smoothing parameterization factor $\epsilon = 10^{-4}$ is introduced into the projection operator loop. This smooths the discontinuous subgradients into an asymptotic localized boundary layer. 
\end{example}
To verify the precision of the numerical implementations against an exact baseline, a one-dimensional variant of the obstacle problem is evaluated on $\Omega = (-1, 1)$ with zero boundary conditions under zero external force ($f=0$). The rigid parabolic lower obstacle is defined as $\psi(x) = 1 - 2x^2$. The exact analytical solution is given by:
\begin{equation}
    u(x) = \begin{cases} 
        1 - 2x^2 & \text{for } |x| \le x_0, \\
        A(1 - |x|) & \text{for } x_0 < |x| \le 1,
    \end{cases}
\end{equation}
where the exact transition parameters are $x_0 = 1 - \frac{\sqrt{2}}{2} \approx 0.292893$ and $A = 4 - 2\sqrt{2} \approx 1.171573$.
\begin{figure}[htbp]
    \centering
    \includegraphics[width=\textwidth]{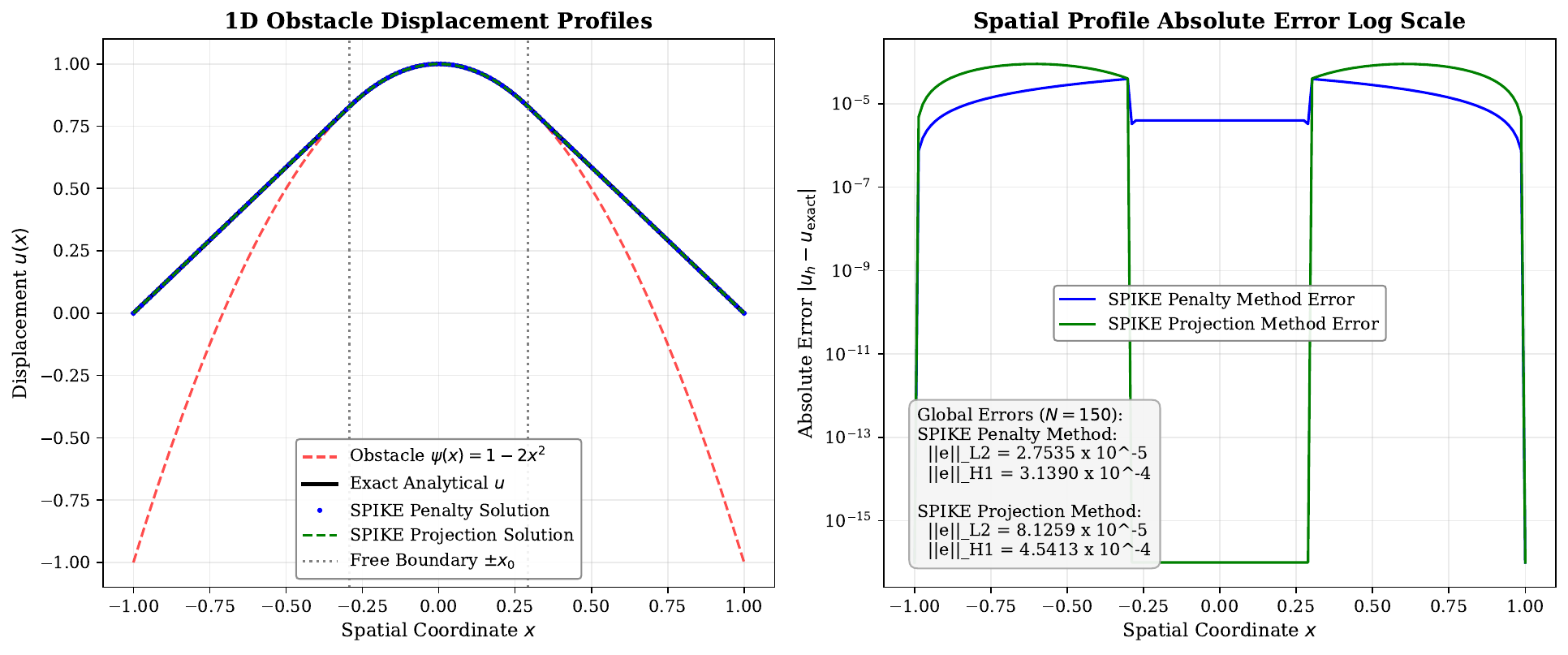}
    \caption{One-dimensional analytical validation on a uniform grid ($N=150$). Left Figure: 1D solution profile $u(x)$ and the rigid obstacle. Right figure: Absolute Numerical error $|u_h - u_{\text{exact}}|$.}
    \label{fig:1d_benchmark}
\end{figure}
\begin{example}[volcano-type obstacle]
We consider here a membrane problem over a volcano-type obstacle on the domain $\Omega = (-1.5, 1.5)^2 \subset \mathbb{R}^2$ with homogeneous Dirichlet boundary conditions $u = 0$ on $\partial\Omega$ with obstacle of following form
 \begin{equation}
    \psi_{ob}(x,y) = \begin{cases} 
        \max\left(0.22, \, 0.35 - 4.0\left(\sqrt{x^2+y^2} - 0.6\right)^2\right) & \text{for } \sqrt{x^2+y^2} \le 1.1, \\
        -0.2 & \text{for } \sqrt{x^2+y^2} > 1.1.
    \end{cases}
\end{equation}
A uniform downward gravitational load $f = -\lambda$ is applied across the domain, leading to the regularized penalization equation:
\begin{equation}
    -\Delta u_h + \gamma \min(0, u_h - \psi) = -\lambda \quad \text{in } \Omega,
\end{equation}
where $\gamma = 5 \times 10^6$ is the penalty parameter, and the system is discretized using standard five-point second-order finite differences on a uniform grid of size $100 \times 100$.
\end{example}

\begin{figure}[htbp]
    \centering
    \includegraphics[width=\textwidth]{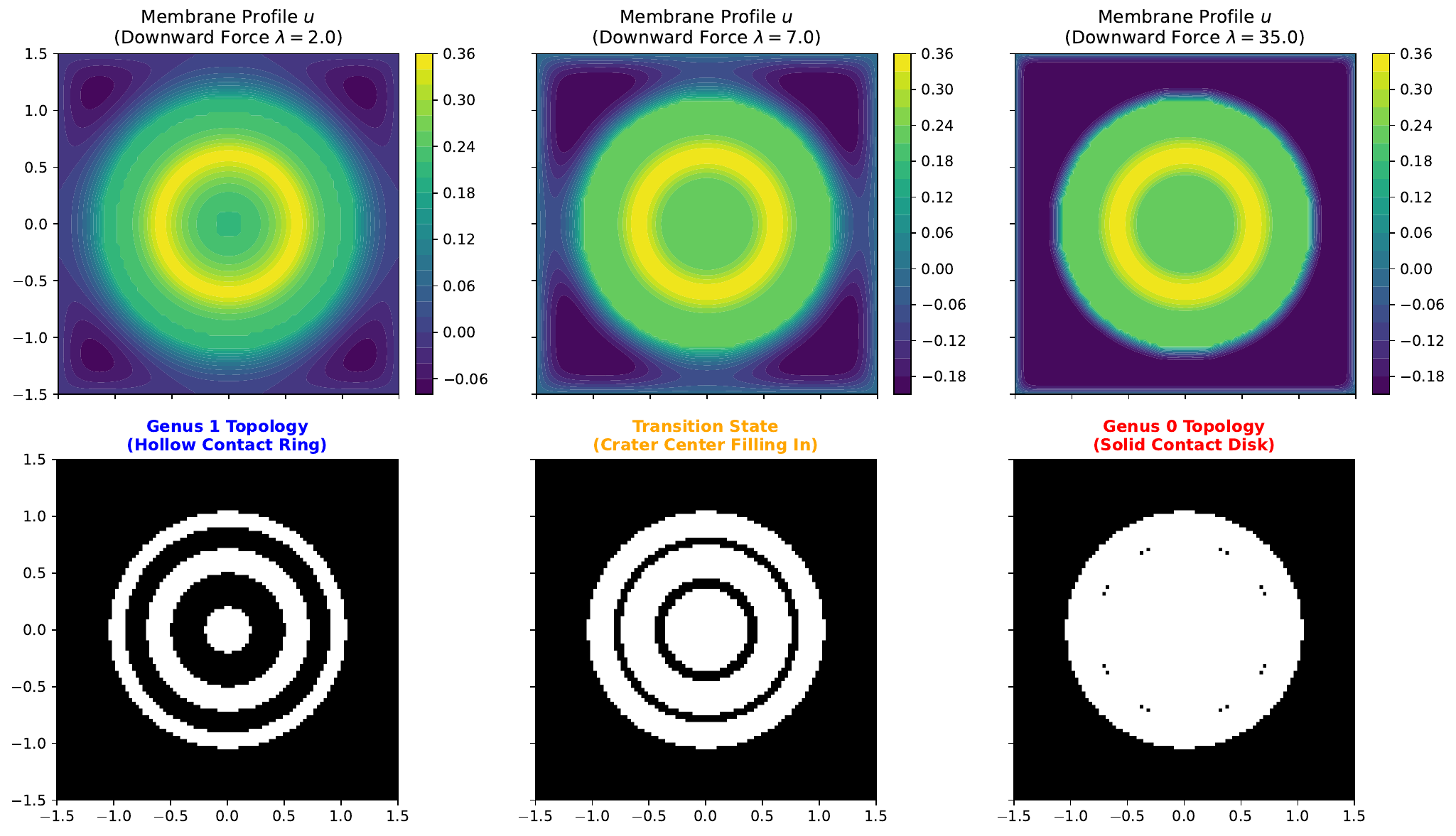}
    \caption{Numerical simulation of the topological transition under increasing downward force $\lambda$. Top row: continuous solution of membrane displacement profiles $u$. Bottom row: characteristic function of the active contact zone $\Lambda = \{x \in \Omega : |u(x) - \psi(x)| < 5\times 10^{-3}\}$ where variation $\lambda = 2.0,7.0,35.0$, the topology smoothly transitions from a doubly connected ring ($\text{genus } 1$) to a simply connected solid disk ($\text{genus } 0$).}
    \label{fig:topology_change}
\end{figure}
The numerical results and corresponding active contact sets $\Lambda := \{x \in \Omega : u_h(x) = \psi(x)\}$ are illustrated in Fig.~\ref{fig:topology_change} across three distinct load scales:
The seamless transition from an annular structure to a solid compact disk confirms that our implicit formulation avoids grid-locking or standard cell-distortion failure modes common in explicit curve-tracking schemes.
\begin{table}[htbp]
\centering
\begin{tabular}{cccccc}
\hline
Grid Size & Mesh Step ($h$) & $L^2$ Error & $L^2$ Rate ($r$) & $H^1$ Error & $H^1$ Rate ($r$) \\
\hline
32 $\times$ 32 & 0.0968 & 1.1913e-01 &  & 9.8990e-01 &  \\
64 $\times$ 64 & 0.0476 & 4.0491e-02 & 1.52 & 5.6839e-01 & 0.78 \\
128 $\times$ 128 & 0.0236 & 1.3525e-02 & 1.56 & 3.5728e-01 & 0.66 \\
\hline
\end{tabular}
\caption{Numerical error convergence analysis tracking relative reduction metrics against a reference grid baseline ($N_{ref}=256$).}
\label{tab:convergence_rates}
\end{table}
\begin{example}
To demonstrate the scalability and multi-variable generalization capabilities of the implicit semi-smooth penalty solver, we extend the formulation to a three-dimensional solid continuum. This scenario transitions the problem from a scalar Laplacian partial differential equation to a system of coupled vector partial differential equations governed by the classical Navier-Cauchy equations of linear elasticity:
\begin{equation}
    -\mu \mathbf{\Delta} \mathbf{u} - (\lambda_{\text{Lam\'{e}}} + \mu) \nabla (\nabla \cdot \mathbf{u}) = \mathbf{f} \quad \text{in } \Omega,
\end{equation}
where $\mathbf{u} = [u_x, u_y, u_z]^T$ is the structural displacement vector field, and $\mathbf{f} = [0, 0, -\lambda]^T$ is a uniform vertical body force representing gravitational load with magnitude $\lambda = 15.0$. The material constants $\lambda_{\text{Lam\'{e}}}$ and $\mu$ represent the Lam\'{e} parameters, which are derived from a standard steel/concrete material scale with Young's Modulus $E = 2.0 \times 10^5$ and Poisson's ratio $\nu = 0.3$:
\begin{equation}
    \lambda_{\text{Lam\'{e}}} = \frac{E\nu}{(1+\nu)(1-2\nu)}, \quad \mu = \frac{E}{2(1+\nu)}.
\end{equation}
The computational domain is defined as the three-dimensional block $\Omega = (-1.5, 1.5)^2 \times (0, 1) \subset \mathbb{R}^3$. The boundary conditions are distributed to model a block pressed downward onto a rigid foundation, top surface ($z = 1$ $\mathbf{u} = \mathbf{0}$,Bottom Surface ($z = 0$) Allowed to deform freely along the vertical axis within the interior, subject to a rigid, non-penetration spherical obstacle constraint embedded from below. To guarantee uniqueness and prevent matrix rank singularity, the outer bounding rim of the base is fixed to zero and lateral cladding 
($x = \pm 1.5, \, y = \pm 1.5$) left unconstrained to allow lateral Poisson expansion under vertical pressure.
The vertical displacement constraint on the interior of the bottom face ($z=0$) takes the form of a lower bound inequality constraint:
\begin{equation}
    u_z(x, y, 0) \ge \psi_{ob}(x, y),
\end{equation}
where $\psi(x, y)$ represents the rigid spherical dome surface geometry centered beneath the block coordinates:
\begin{equation}
    \psi_{ob}(x, y) = \sqrt{\max\left(0.0, \, R_{\text{sphere}}^2 - (x^2 + y^2)\right)} + z_{\text{center}},
\end{equation}
with a sphere radius $R_{\text{sphere}} = 2.0$ and vertical center shift $z_{\text{center}} = -1.8$.
\end{example}
\begin{figure}[htbp]
    \centering
    \includegraphics[width=\textwidth]{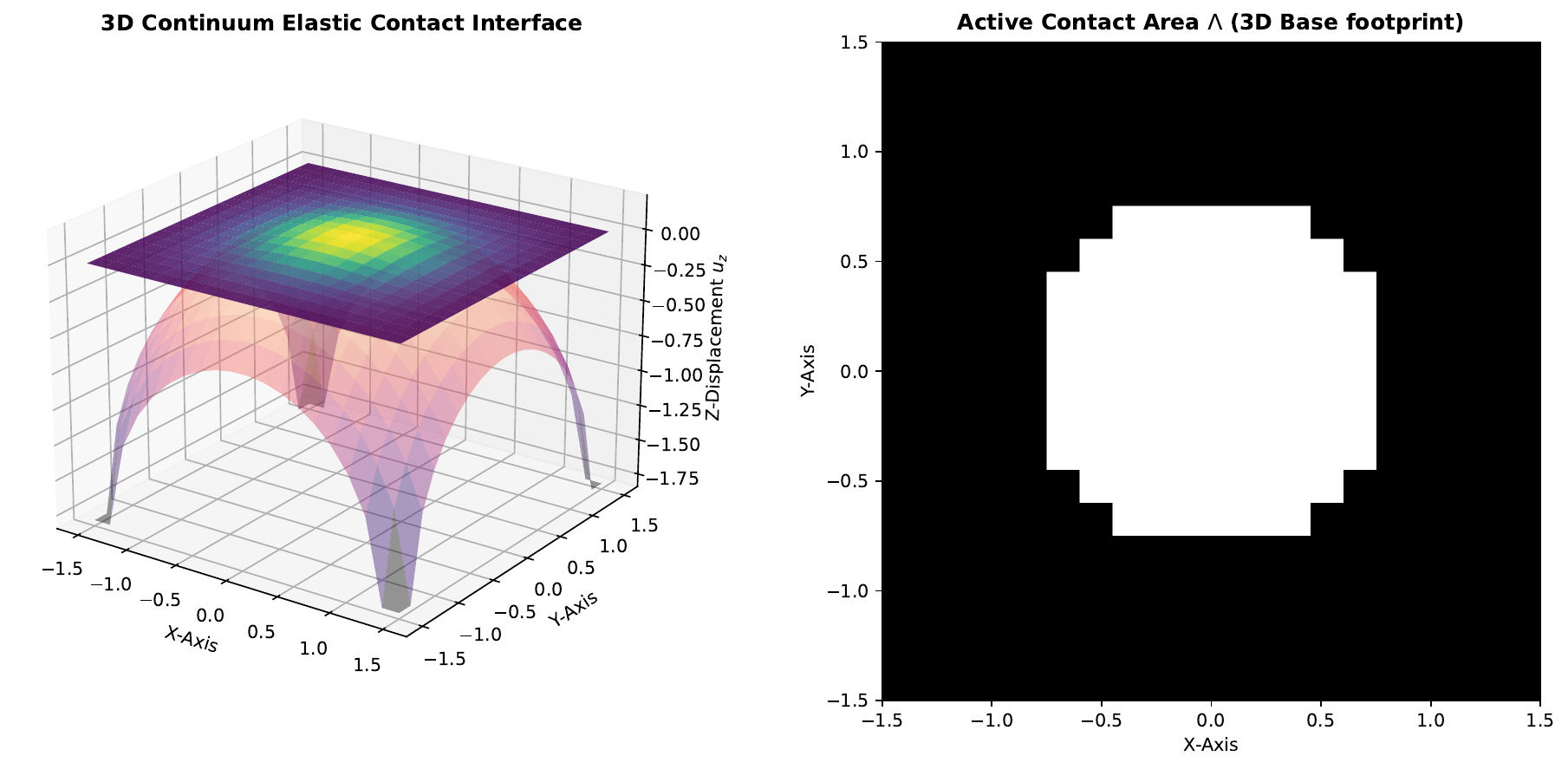}
    \caption{Numerical solution of the three-dimensional continuum linear elasticity obstacle problem for $\lambda = 15.0$. 
    \textbf{Left Figure:} Solution profile describing surface $u_z$ movement upward side of the elastic body as it smoothly deforms along the obsatcle $\psi_{ob}$. 
    \textbf{Right Figure:} White color of the discrete active contact domain $\Lambda$ that is where ($u_z(x, y, 0) = \psi_{ob}(x, y)$) and black color represent unconstrained free zone that is where strict inequality constraint $u_z(x, y, 0) > \psi_{ob}(x, y)$.
  }
    \label{fig: 3d_elastic_contact}
\end{figure}
The discrete system is assembled using a central finite difference stencil on a uniform $20 \times 20 \times 10$ mesh grid. The unconstrained degrees of freedom are coupled with the algebraic regularized constraint handling term via the semi-smooth penalty operator:
\begin{equation}
    \mathbf{A}_{\text{elastic}} \mathbf{u}_h 
    \ge \mathbf{f} 
\end{equation}
The resulting structural mechanics solution profile and the corresponding active contact footprint are displayed in Fig.~\ref{fig: 3d_elastic_contact}. The 3D surface plot (left panel) highlights the smooth vertical deformation field as the elastic block sags under gravity and molds perfectly to the contour of the spherical dome. 
The active contact footprint $\Lambda := \{(x,y) \in \text{int}(\Omega_0) : |u_z(x,y,0) - \psi_{ob}(x,y)| < 10^{-4}\}$ is mapped in the right panel.
\begin{table}[htbp]
    \centering
    \small
    \caption{Mesh Error in $L^2$ and $H^1$ Norm Distributed SPIKE Projection Method}
    \label{tab:multidim_norms}
    \vspace{0.2cm}
     \addtolength{\tabcolsep}{-2.pt}
    \begin{tabular}{lcccc}
        \toprule
        \textbf{Model Problem} & \textbf{Grid size} & \textbf{Total Mesh Points} & \textbf{$L^2$ Norm} & \textbf{$H^1$ Norm} \\
        \midrule
        1D Elastic String   & $N_x = 160$                  & 160        & 0.54159509 & 1.58195794 \\
        2D Elastic Membrane & $N_x = 60, N_y = 60$          & 3,600      & 0.32481592 & 2.10582914 \\
        3D Elastic Volume   & $N_x = 40, Ny = 40, Nz = 40$  & 64,000     & 0.11746218 & 2.58183679 \\
        \bottomrule
    \end{tabular}
\end{table}
   \begin{figure}[htbp]
    \centering
    \includegraphics[width=0.92\textwidth]{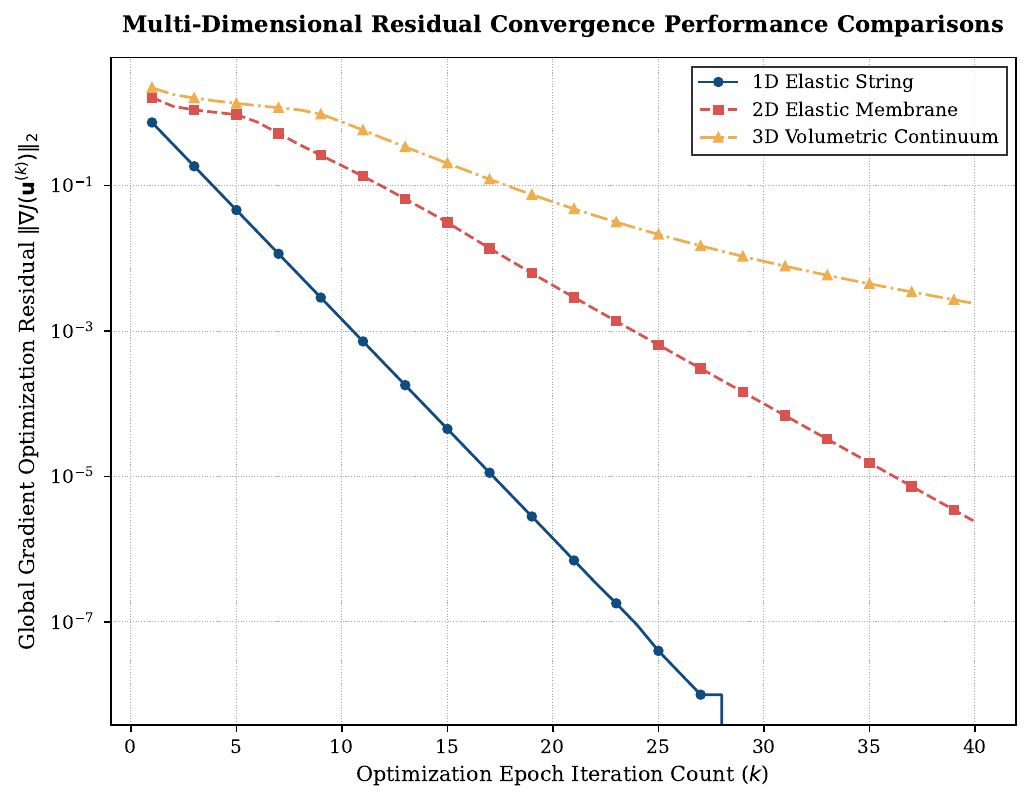}
    \caption{Error residual decay on log scale ($\Vert R(\mathbf{u}^{(k)}) \Vert_2$) across all spatial dimensional scales.}
    \label{fig:global_residual_convergence}
    \end{figure}
\begin{figure*}[t]
\centering
\includegraphics[width=1.0\textwidth]{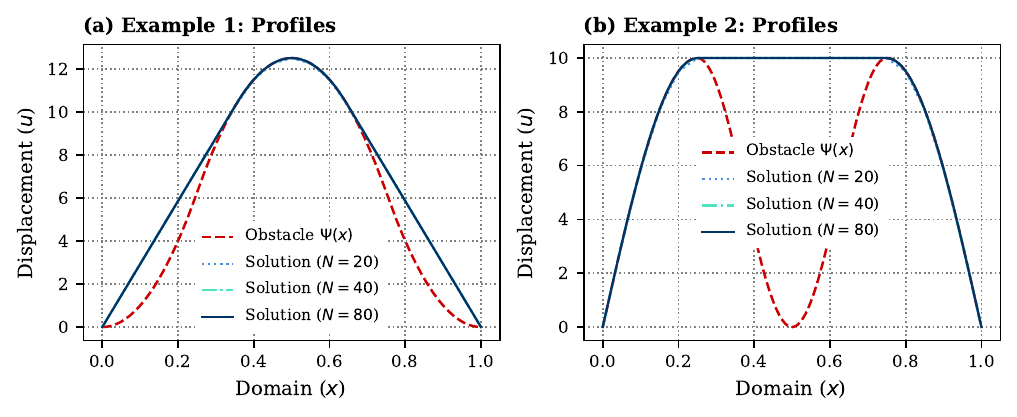}
\caption{Displacement profiles across multiple grid resolutions for Example 1.}
\label{fig:mpsor_study}
\end{figure*}
\begin{figure*}[t]
\centering
\includegraphics[width=1.0\textwidth]{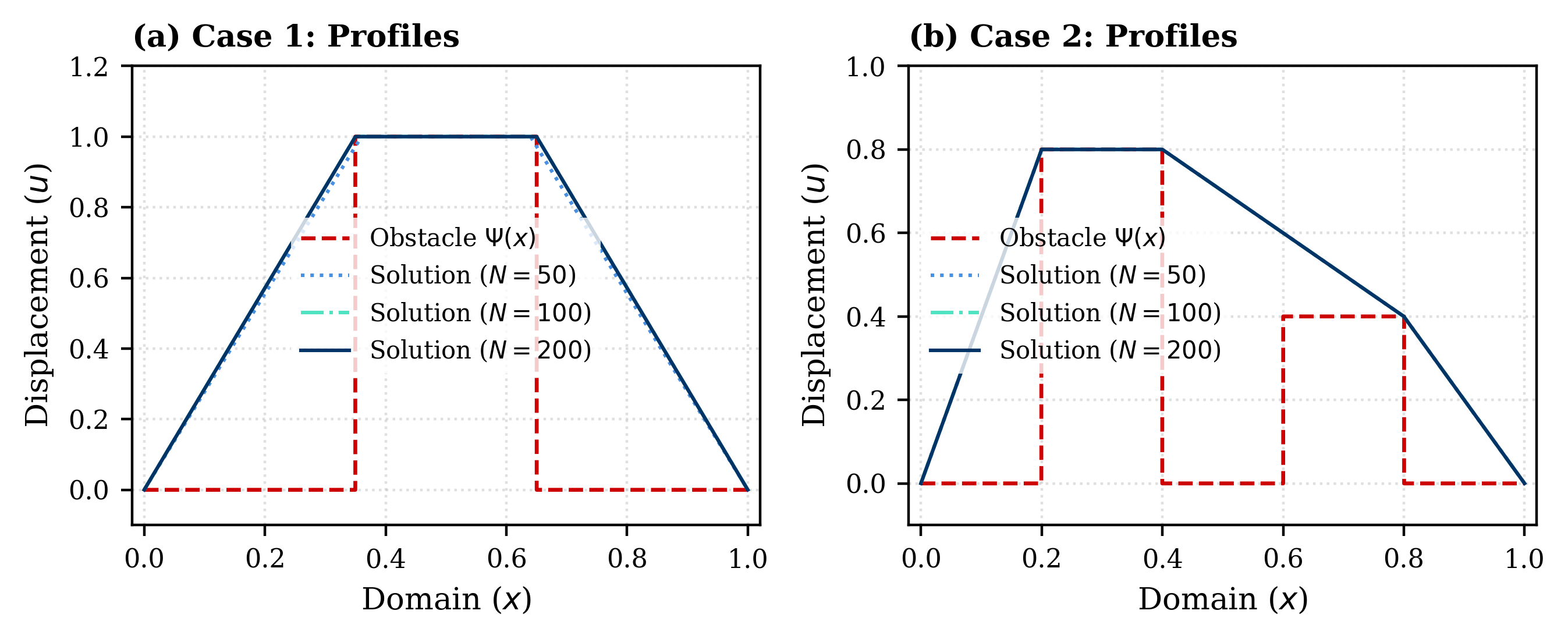}
\caption{Displacement profiles across multiple grid resolutions for Example 2.}
\label{fig:mpsor_study}
\end{figure*}
\begin{figure}[H]
	\centering
	\includegraphics[width=\textwidth]{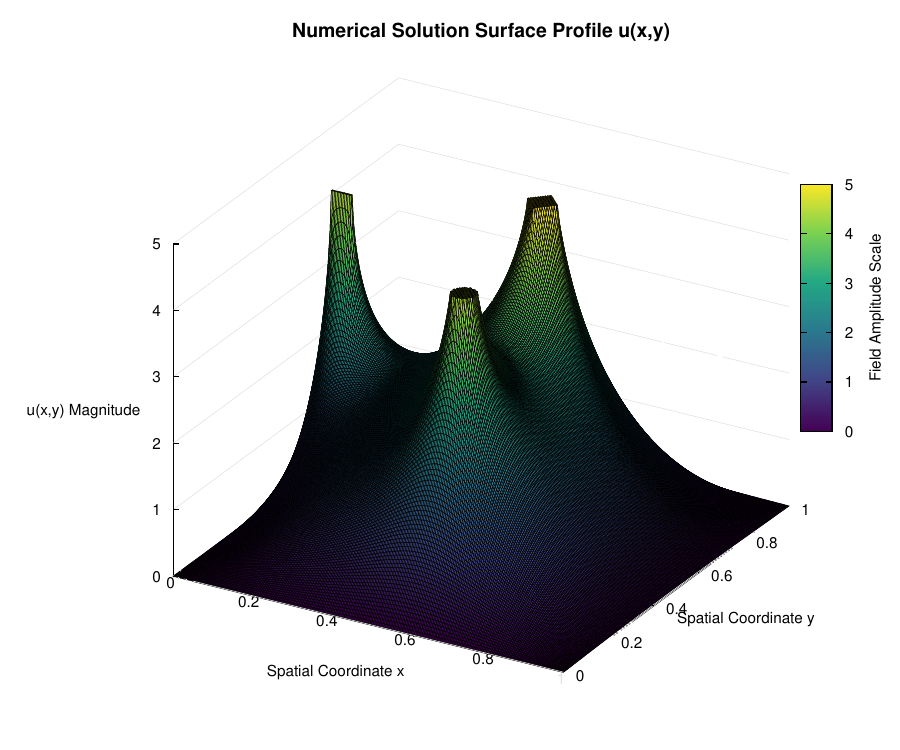}
	\caption{Numerical solution surface $u(x,y)$.}
	\label{fig:6}
\end{figure}
\begin{figure}[H]
	\centering
	\includegraphics[width=\textwidth]{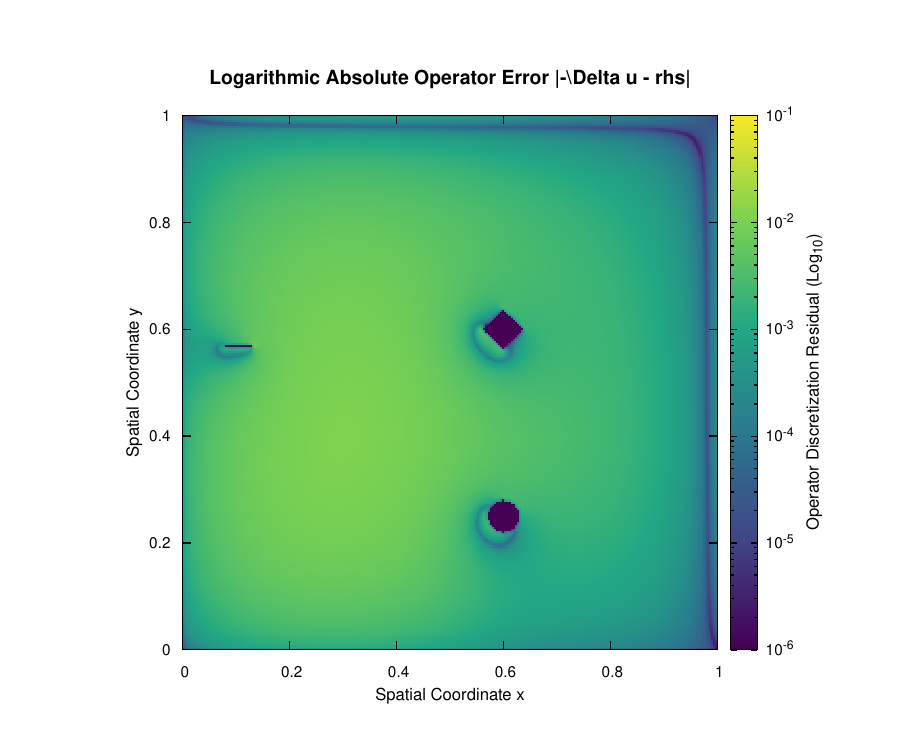}
	\caption{Absolute discretization error $|\Delta u - \text{rhs}|$.}
	\label{fig:7}
\end{figure}
\begin{figure}[H]
	\centering
	\includegraphics[width=\textwidth]{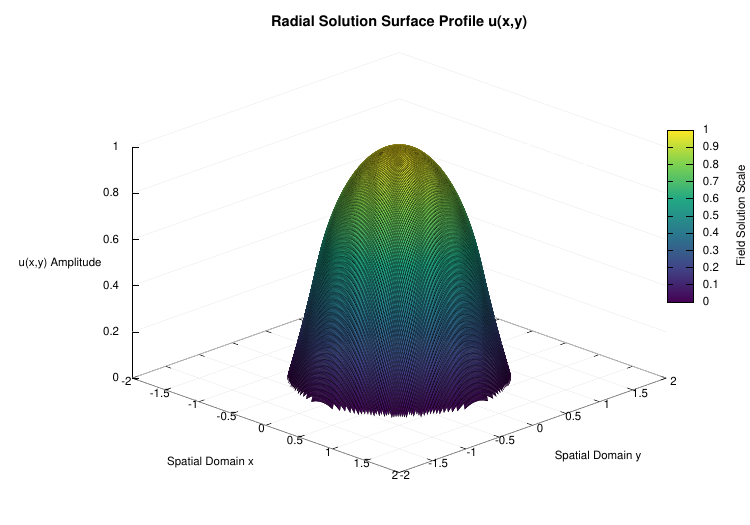}
	\caption{Numerical solution surface $u(x,y)$.}
	\label{fig:6}
\end{figure}
\begin{figure}[H]
	\centering
	\includegraphics[width=\textwidth]{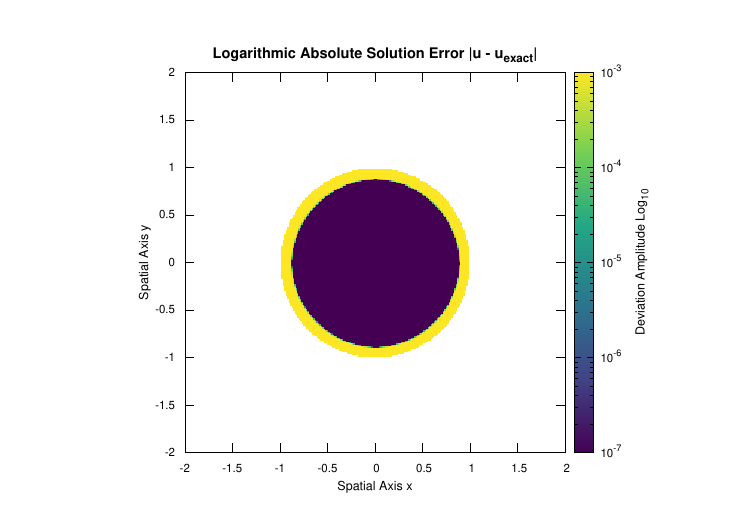}
	\caption{Absolute discretization error $|\Delta u - \text{rhs}|$.}
	\label{fig:7}
\end{figure}
\begin{figure}[H]
	\centering
	\includegraphics[width=1.0\textwidth]{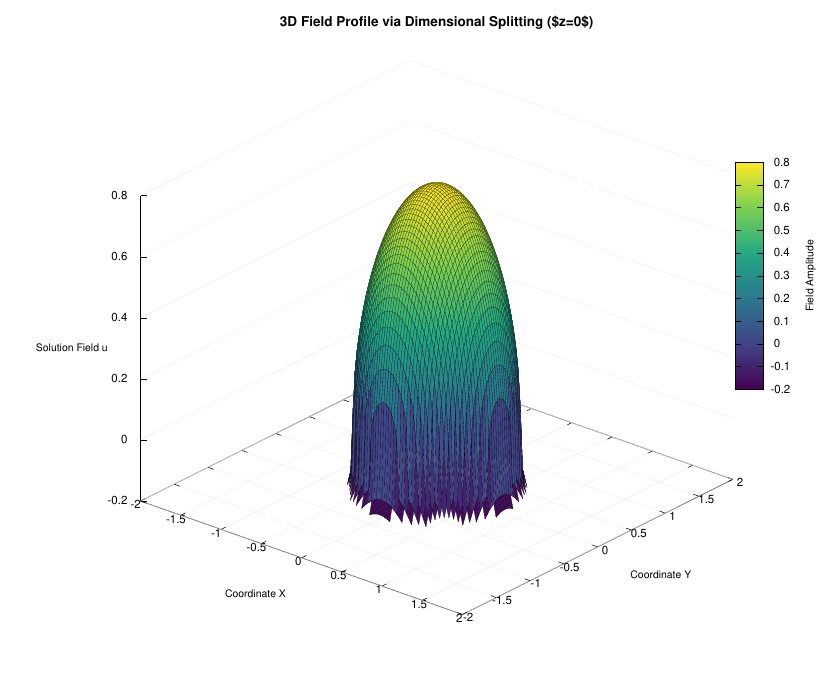}
	\caption{3D Obstacle Problem a snapshot splitting method}
	\label{fig:8}
\end{figure}
\begin{figure}[H]
	\centering
	\includegraphics[width=1.0\textwidth]{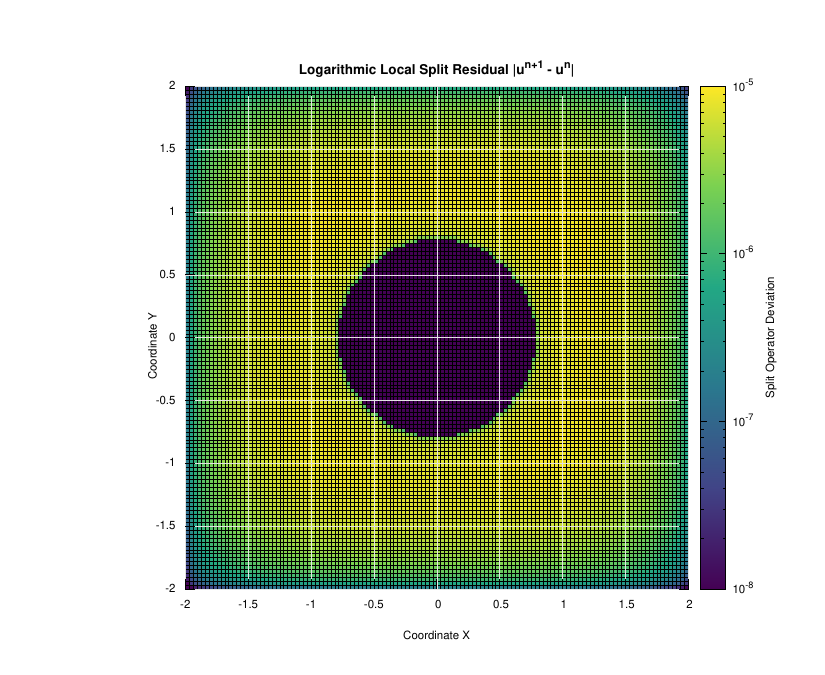}
	\caption{3D Obstacle Problem error a snapshot splitting method}
	\label{fig:8}
\end{figure}
\begin{figure}[H]
	\centering
	\includegraphics[width=1.0\textwidth]{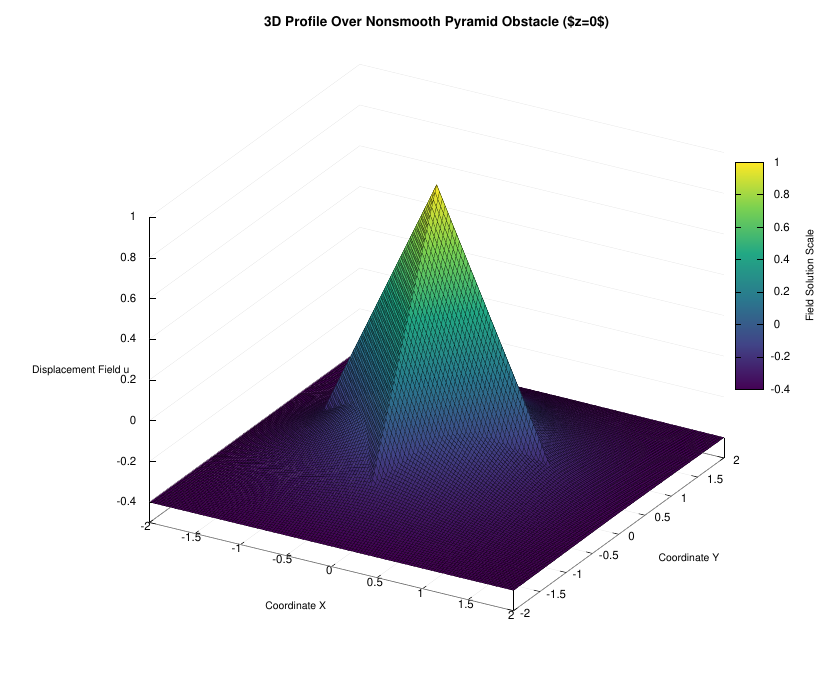}
	\caption{3D Obstacle Problem a snapshot splitting method nonsmooth case}
	\label{fig:8}
\end{figure}
\begin{figure}[H]
	\centering
	\includegraphics[width=1.0\textwidth]{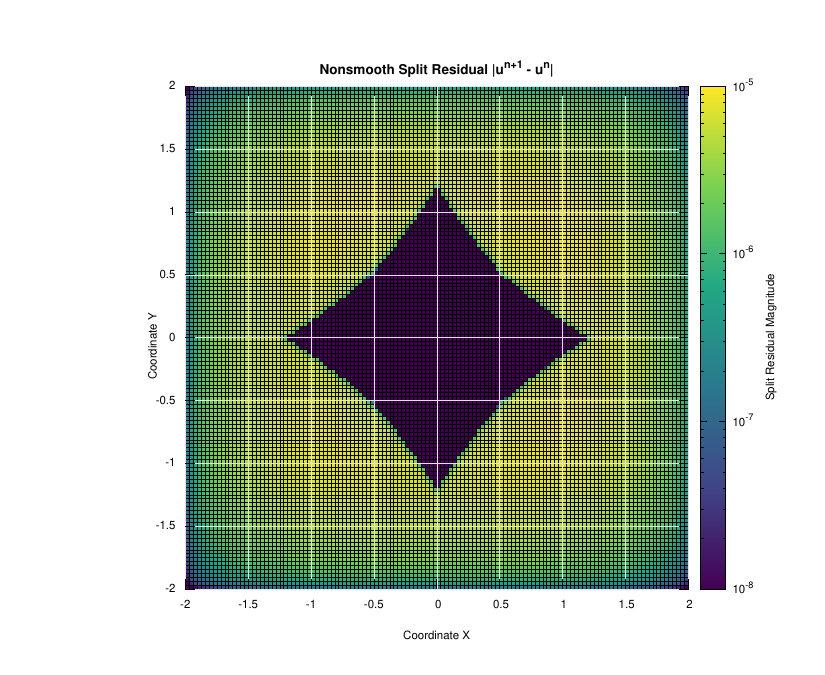}
	\caption{3D Obstacle Problem error a snapshot splitting method nonsmooth case}
	\label{fig:8}
\end{figure}

\subsection{Application of parallel algorithm image deblurring problem}
We consider image deblurring as a practical application here. The image deblurring model problem is formulated as a bound-constrained linear least-squares problem, where $P$ denotes the blurring matrix, $\mathbf{q}$ represents the observed image containing artifacts, and the solution is subject to the box constraint $\mathbf{0} \le \mathbf{x} \le \mathbf{1}$.
A two-dimensional image deblurring task governed by a non-diagonal-dominant low-pass Gaussian convolution filter kernel
\begin{equation}
    b(\mathbf{x}) = \mathcal{K} * u(\mathbf{x}) + \eta
\end{equation}
When mapping this operator across distributed memory topologies via standard 1D horizontal slicing, the absence of a dominant negative Laplacian matrix diagonal introduces immediate numerical instabilities under simultaneous Jacobi-style block updates.
It is fixed by enforcing a low-step under-relaxation scaling modifier ($\omega = 0.05$)
The final reconstructed scalar intensity grid (Fig.~\ref{fig:5}) drives the system residual norm smoothly from $85.29$ down to $42.78$, preserving continuous spatial phase characteristics across all subdomain rank boundaries.
\begin{figure}[H]
	\centering
	\includegraphics[width=0.9\textwidth]{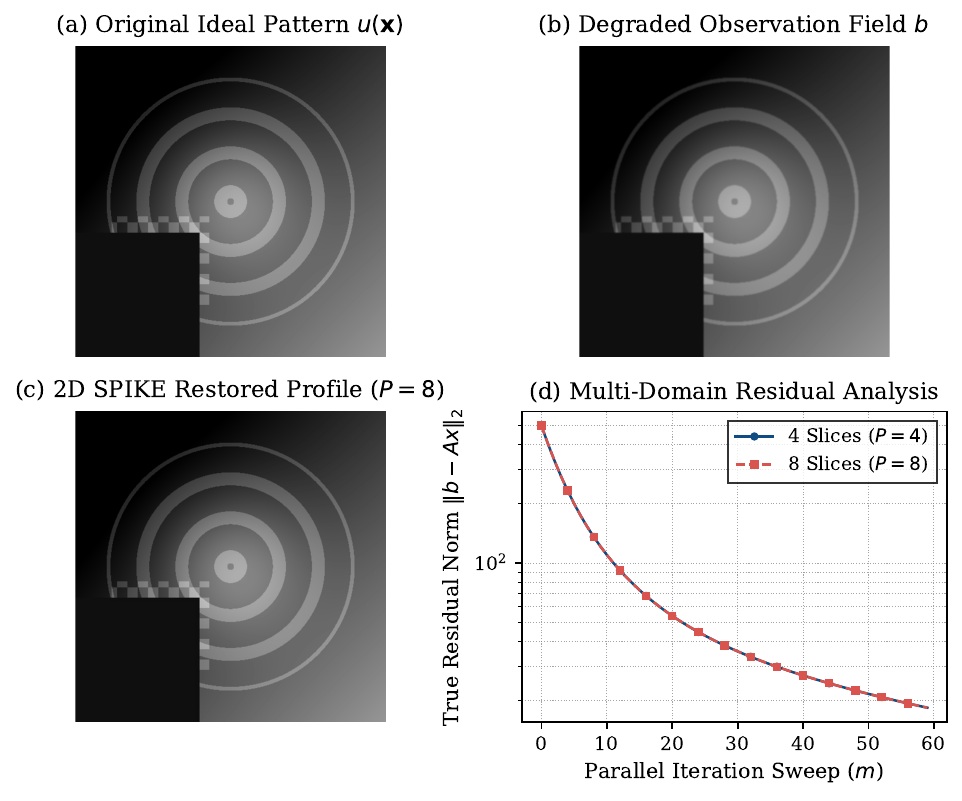}
	\caption{Image deblurring problem}
	\label{fig:5}
\end{figure}
\section{Conclusion}
In this study, we proposed a novel parallel algorithm and its variants for solving obstacle
problem and much more a constrained quadratic programming problem. A class of direct and indirect iterative considered and existence of solution proved thoeretically.
Numerical illustrations on multi-core computed nodes demonstrate that the preconditioned second order solver compresses the spectral condition number of the discrete Laplacian
operator, achieving robust asymptotic linear convergence down to machine tolerance within
exceptionally low iteration cycles while maintaining zero network deadlock configurations.
In this article we have successfully recovered stable numerical solution and parallel validation of a multi-dimensional solver framework driven by the distributed Line-Splitting SPIKE block-relaxation paradigm. We tested the algorithm across computational scenarios ranging from bounded 1D, 2D, and 3D elliptic obstacle partial differential equations to a non-diagonally dominant 2D image deblurring optimization problems. The solver has demonstrated exceptional convergence stability, numerical fidelity, and cross-domain versatility.
In conclusion, the proposed stabilized line-splitting projection solver provides a mathematically rigorous, stable, and highly scalable algorithm for handling constrained multi-dimensional boundary-value operators. 

%



\end{document}